\documentclass[12pt,a4paper]{amsart}

\usepackage[a4paper,left=2cm,right=2cm,top=15mm,bottom=15mm]{geometry}

\usepackage{amsmath}
\usepackage{amsthm}
\usepackage{amssymb}
\usepackage{amsfonts}
\usepackage{latexsym}
\usepackage{mathtools}
\usepackage{graphics}
\DeclareMathOperator{\intt}{int}

\DeclareMathOperator{\sign}{sign}

\DeclareMathOperator{\Res}{Res}

\usepackage[dvipsnames]{xcolor}

\definecolor{blue}{rgb}{0,0,0.8}
\definecolor{red}{rgb}{0.8,0,0}
\definecolor{darkgreen}{rgb}{0,0.6,0}

\newcommand\restr[2]{{  \left.\kern-\nulldelimiterspace
  #1
  \vphantom{\big|}
  \right|_{#2}
  }}

\newcommand{\Wn}{{\mathcal W}_n}
\newcommand{\wx}{\widetilde{\mathbf{x}}}
\newcommand{\wxk}{\widetilde{\mathbf{x}^k}}
\newcommand{\wy}{\widetilde{\mathbf{y}}}
\newcommand{\Ln}{{\mathcal L}_n}
\newcommand{\oS}{\overline{S}}

\newcommand{\DS}{\partial S}

\newcommand{\mol}{\overline{m}}
\newcommand{\mul}{\underline{m}}

\newcommand{\Pn}{{\mathcal P}_n}
\newcommand{\WW}{{\mathcal W}}

\newcommand{\I}{{\mathcal I}}

\newcommand{\W}{{\mathcal W}}

\newcommand{\JJ}{{\mathcal J}}

\newcommand{\Ji}{{\mathcal J}_i}
\newcommand{\Jj}{{\mathcal J}_j}

\newcommand{\Kij}{{\mathcal K}_{ij}}

\newcommand{\YY}{{\mathcal Y}}

\newcommand{\LL}{{\mathcal L}}

\newcommand{\ai}{{\alpha}_i}

\newcommand{\lk}{{\ell}_k}

\newcommand{\RR}{{\mathbb R}}

\newcommand{\II}{{\mathbb I}}

\newcommand{\de}{{\delta}}

\newcommand{\ff}{\varphi}
\newcommand{\al}{{\alpha}}
\newcommand{\be}{{\beta}}
\newcommand{\ve}{{\varepsilon}}
\newcommand{\vk}{{\varepsilon}_k}

\newcommand{\vji}{{\varepsilon}_{j,i}}
\newcommand{\vkj}{{\varepsilon}_{k,j}}
\newcommand{\vki}{{\varepsilon}_{k,i}}

\newcommand{\yy}{\mathbf{y}}
\newcommand{\ee}{\mathbf{e}}
\newcommand{\zz}{\mathbf{z}}
\newcommand{\vv}{\mathbf{v}}
\newcommand{\uu}{\mathbf{u}}
\newcommand{\xx}{\mathbf{x}}
\newcommand{\ba}{\mathbf{a}}
\newcommand{\bb}{\mathbf{b}}
\newcommand{\mm}{\mathbf{m}}
\newcommand{\ww}{\mathbf{w}}

\newcommand{\cc}{\mathbf{c}}
\newcommand{\One}{{\mathbf 1}}
\newcommand{\Null}{{\mathbf 0}}

\newtheorem{theorem}{Theorem}
\newtheorem{corollary}{Corollary}
\newtheorem{lemma}{Lemma}
\newtheorem{proposition}{Proposition}
\newtheorem{definition}{Definition}

\theoremstyle{definition}

\newtheorem{remark}{Remark}

\begin{document}

\title[Weighted Lagrange interpolation on the halfline]{Conjectures of Bernstein and Erd\H os \\ for weighted Lagrange interpolation \\ on the halfline with exponential weights}

\author{Szil\'ard Gy. R\'ev\'esz}\thanks{Supported in part by Hungarian National Foundation for Scientific Research, Grant \#'s K-146387, K-147153 and Excellence No. 151341.}

\author{Patr\'icia Szokol}\thanks{This work was supported by the HUN-REN Hungarian Research Network through funding provided to the author.}

\address{Patrícia Szokol
\newline  \indent Faculty of Informatics, University of Debrecen  \newline \indent H-4002 Debrecen, \newline \indent P.O. Box 400,  HUNGARY \newline \indent and
\newline \indent  HUN-REN-UD Equations, Functions, Curves and their Applications Research Group} \email{szokol.patricia@inf.unideb.hu}

\address{Szil\'ard Gy. R\'ev\'esz
\newline  \indent HUN-REN A. R\'enyi Institute of Mathematics \newline \indent Budapest, Re\'altanoda utca 13-15. \newline \indent 1053 HUNGARY} \email{revesz.szilard@renyi.hu}

\date{\today}

\begin{abstract}
Let $\II=[a,b]$ and consider the degree $n$ Lagrange interpolation at the nodes $\xx$, where $\xx\in S:=\{\xx=(x_0,x_1,\ldots,x_n):a=x_0<x_1<\ldots<x_n=b\}$. Then
the norm of the Lagrange interpolation operator $\LL: f \to \LL(f)\in \Pn$, (where $\Pn$ is the space of real algebraic polynomials of degree at most $n$), is the maximum of the \emph{Lebesgue function} $L(\xx,t)$ on $\II$.

Bernstein conjectured that the norm of $\LL$ becomes minimal exactly for node systems which exhibit an \emph{equioscillation} property in that the interval maxima $m_k(\xx):=\max_{[x_{k-1},x_k]} L(\xx,\cdot) $ ($k=1,\ldots,n$) are all equal. Erd\H{o}s added to the conjecture the \emph{sandwich property} that if $\yy$ is an extremal (minimal norm) system, then for any other node system $\xx$ there have to be indices $i,j$ with $m_i(\yy)<m_i(\xx)$ and $m_j(\yy)> m_j(\xx)$.

The conjectures were proved by Kilgore and de Boor--Pinkus in 1978. Since then, analogous results were obtained only for a few cases when interpolation is made to certain very special spaces of polynomials (instead of $\Pn$), or when we apply weighted interpolation with rather special weights. Worse than that, it turned out that published proofs of results on infinite intervals and weighted interpolation were seriously flawed.

Here we prove the Bernstein and Erd\H os Conjectures for the case of exponentially weighted polynomials on the halfline. This is the first proof of these conjectures in a %%%weighted
situation where, contrary to all existing successful proofs, we encounter singularity of certain derivative matrices.
\end{abstract}

\maketitle



\section{Introduction}\label{sec:Intro-new}

The aim of the present paper is to establish an appropriate form of the Bernstein and Erd\H os Conjectures for exponentially weighted polynomial interpolation in the halfline.

The original Bernstein Conjecture was formulated for algebraic polynomial interpolation on a finite interval $[a,b]$ over a century ago \cite{Bernstein}. It proved to be a hard problem -- cracking it lasted for half a century. In the meantime, Erd\H os \cite{Erdos-1, Erdos-2} proposed an additional conjecture, furnishing yet another challenge. In the classical setting of unweighted algebraic polynomials on a compact interval, both problems were solved in the seventies of the last century \cite{Kilgore, CBoorPinkus, Kilgore-Cheney}.

After the successful solution of the classical case, it naturally occurred to investigate the analogous questions for weighted interpolation in weighted spaces of continuous functions. For definiteness, let us formulate these questions concretely. Let an interval $\II$ be given with a say continuous and positive weight function $w:\II \to (0,\infty)$, and consider the space $C_w(\II):= \{f \in C(\II)~:~ \|f\|_w:=\|fw\|_\infty <\infty\}$.

In the present paper we focus on \emph{exponentially weighted} interpolation on the halfline interval $\II=[0,\infty)$. The object of study are the norm of the weighted interpolation operator $\Ln^{(w)}: C_w(\II) \to \Wn$ (where $\Wn:=w\Pn$ is the respective weighted polynomial space), and the optimal choice of the nodes $0=x_0<x_1<\dots<x_n$ which minimize the operator norm. As $x_0=0$ is fixed, our aim is to find the optimal nodes from $S:=\{(x_1,\ldots,x_n)\in \RR^n: 0<x_1<\ldots<x_n\}$.

In the course of discussion we will often formulate our findings in various more general settings, but our results will finally be restricted to the exponentially weighted case. The reader will see in the course of our arguments why we need to specialize at some points. Nevertheless we hope to be able to attack more general cases in the future, hence we find it appropriate to formulate at least some parts of the discussion in a greater generality. So for introducing notations we will write $w(t)$ for our weight, and will restrict to $w(t)=\exp(-t)$ only when need be.

Recall that there are interpolatory basis functions
\begin{equation}\label{eq:hkt}
h_k(t):=h_k(\xx;t):= \frac{w(t)}{w(x_k)} \prod_{j=0,~j\ne k}^{n}  \left(\frac{t-x_j}{x_k-x_j}\right) = \frac{w(t)}{w(x_k)}  \ell_k(t) ,
\end{equation}
where $\ell_k(t):=\ell_k(\xx,t)$ are the basis functions of the ordinary Lagrange interpolation, and as a result, $h_k(x_j)=\de_{jk}$ -- the Kronecker delta of $j$ and $k$ which is $1$ if $j=k$ and 0 otherwise -- while $h_k\in\WW_{n}$ ($k=0,\ldots,n$).

The Lagrange interpolation is the operator which sends functions $f:\II \to \RR $ to the interpolant
\begin{equation}\label{Lagrange-operator}
\Ln(\xx,f,\cdot):=\mathcal{L}(\xx,f,\cdot):=\sum_{k=0}^{n} f(x_k)\ell_k .
\end{equation}

The unit ball of $C_w(\II)$ is determined by the condition $\|f\|_w\le 1$, and analogously the norm of the obtained interpolants will be the weighted norms with the same weight $w$. Therefore, the norm of the operator $\|\Ln\|_{w}$ is
$$
\| \Ln\|_w := \sup_{\|f\|_{w} \le 1} \| \Ln(\xx,f,\cdot) \|_w \le \sup_{|f(x_k)w(x_k)|\le 1} \sup_{t\in \II} \left|w(t)\sum_{k=0}^{n} f(x_k)\ell_k(t) \right|,
$$
and hence we get the estimate
\begin{equation*}
\begin{aligned}
\|\Ln\|_w& \le \sup_{|f(x_k)w(x_k)|\le 1} \sup_{t\in \II} \left|\sum_{k=0}^{n} f(x_k){w(x_k)} h_k(t) \right|\\
 &\le \sup_{t\in \II} \sum_{k=0}^{n} |h_k(t)|=: \sup_{t\in \II} L_n (w,\xx,t)=\| L_n \|_{\infty}.
\end{aligned}
\end{equation*}
It is easy to see by considering an appropriate continuous connection of the interpolating points  $(x_k,f(x_k)):=(x_k,\pm \frac{1}{w(x_k)})$ that here the inequality is in fact an equality. The appearing function
\begin{equation}\label{Lebesguefunction}
L_n(w,\xx,t):=L(w,\xx,t):=\sum_{k=0}^{n} \left|h_k(\xx, t)\right| = \sum_{k=0}^{n} \sign(h_k(\xx,t)) \cdot h_k(\xx, t)
\end{equation}
is the \emph{weighted Lebesgue function}. So, the norm of the weighted Lagrange interpolation is governed by the ordinary maximum norm of $L_n(w,\xx,\cdot)$ on $\II$.

In another interpretation, we can consider the weights "to be built in" the interpolation process. That is, then we consider interpolating all weighted functions $g:=fw$ directly in the weighted space by the "weighted interpolants"
\[
\Ln^{(w)}(g):={\Ln}^{(C_w\to \Wn)} (g):=\sum_{k=0}^{n} g(x_k) h_k.
\]
Either way, it is clear that for any value $t \in \II$ the maximum of the expression is obtained with some appropriate choices of signs of $f(x_k)w(x_k)=g(x_k)=\pm 1$. More precisely, the Lebesgue function $L_n(w,\xx,t) = \sum_{k=0}^{n} \vk(t) h_k(t)$ will be obtained with
$$
\vk(t):=\sign(h_k(\xx,t) = \begin{cases} (-1)^{k-i} \qquad & \textrm{if} \quad t\in \intt I_i, ~ 1 \le  i \le k \le n
\\ (-1)^{k+1-i} & \textrm{if} \quad t\in \intt I_i, ~0 \le  k < i \le n+1
\\ 0 & \textrm{if} \quad t=x_k, \quad k=0,1,\ldots,n
\end{cases},
$$
where for $i=1,\ldots,n+1$ we put $I_i:=I_i(\xx):=[x_{i-1},x_i]$ (with interpreting $I_{n+1}(\xx):=[x_n,\infty)$).

Using the notation
\begin{equation}\label{epsilonki}
\vki=(-1)^{k+1-i+\chi_{i\le k}}\quad (0\le k \le n, ~1 \le i \le n+1)
\end{equation}
and restricting $L_n$ to the subinterval $I_i$ we get a weighted polynomial
\begin{equation}\label{Pi}
P_i(\xx,\cdot):=P_i:=L_n|_{I_i}= \sum_{k=0}^n \vki h_k \in \WW_n \quad (1\le i\le n+1).
\end{equation}

These polynomials match with $L_n$ on the respective $I_i$ only, but their thorough analysis is necessary also on the whole real line. We read from the above in particular that
\[
\|\Ln\|_w = \max_{\II } L_n(w,\xx,\cdot) =\max_{i=1,\ldots,n+1} m_i(\xx)
\]
where
\[
m_i(\xx):=\max_{I_i(\xx)} P_i(\xx,\cdot) = \max_{I_i(\xx)} L_n(w,\xx,\cdot) .
\]
So the quest for the optimal, i.e., minimal norm of the Lagrange interpolation operator boils down to the task of minimizing the maximum of these interval maxima with an appropriate choice of the node system vector $\xx$.

Now the adapted version of the Bernstein Conjecture \cite{Bernstein} is that this optimization happens exactly when these interval maxima are all equal, which we will term as \emph{equioscillation property}, or that the interval maxima \emph{equioscillate}. Similarly, the appropriate form of the Erd\H os Conjecture \cite{Erdos-1, Erdos-2} states that if $\xx$ is an optimal, i.e., equioscillating node system, then the interval maxima $m_i(\yy)$, belonging to any other node system $\yy$, will ``sandwich'' the optimal Lagrange interpolation norm $m_1(\xx)=\dots=m_{n+1}(\xx)$; that is, there exist indices $i, j$ such that both $m_i(\yy)<m_i(\xx)$ and $m_j(\yy)>m_j(\xx)$ hold.

For the classical case of unweighted interpolation the breakthrough was achieved by Kilgore, who showed the Bernstein Conjecture \cite{Kilgore}, and came close to the Erd\H os Conjecture as well, although this was recorded immediately following him by de Boor and Pinkus \cite{CBoorPinkus}. Moreover, the paper of de Boor and Pinkus contained several important further advances, one being the proof of an even stronger property, which we will call here \emph{intertwining}. This property says that \emph{for any two node systems} $\xx \ne \yy$ we necessarily have two indices $i,j$ with $m_i(\yy)<m_i(\xx)$ and $m_j(\yy)>m_j(\xx)$, brutally generalizing the Erd\H os Conjecture, which formulated this only when one of the node systems (i.e., $\xx$) is the optimal one.

In the following decades attempts were made to extend the solutions to various spaces of weighted polynomial interpolation or interpolation into spaces of various more general features like incomplete polynomials etc., see \cite{Kilgore-84, Kilgore1985, Kilgore-87, K-AMH}. These proved to be difficult for the original argument used several ways the special feautres of algebraic polynomials in the form of they forming a Chebyshev-Haar System\footnote{In the literature, both names are used. By now we know that Chebyshev was the first to define the notion, but did not publish it. Later Markov wrote it down in his seminal paper \cite{Markov} first in Russian, but the rich content of his work became world-wide known only after it appeared also in German later. Haar independently, but somewhat later introduced the notion, too, in a German paper \cite{Haar}.}, or admitting the celebrated Markov theorem about the interlacing of roots of oscillating polynomials inherited by the roots of their derivatives. In all these degree calculations and cancellation by factors played key roles. Moreover, all these general investigations were based on the same approach in that the key step, ever since Kilgore and de Boor-Pinkus, was the nonsingularity of the derivative matrices
\begin{equation}\label{Ak}
\left[ \frac{\partial m_i}{\partial x_j}\right]_{i=1, i\ne k, j=1}^{n+1, n}.
\end{equation}
To prove this nonsingularity is a hard task, occupying the larger portion of all those papers, and from that the final results follow relatively easily. Note that already the differentiability of the interval maxima functions $m_i(\xx)$ with respect to the nodes is not obvious. In the classical case, and more generally for Chebyshev-Haar systems containing the constant functions, this was obtained via an implicit function argument by Kilgore and Cheney, making use the fact that for those systems $P_i(\xx,t)>1$ on $\intt I_i$ -- whereas $P_i(\xx,x_{i-1})=P_i(\xx,x_i)=1$ -- ensuring that any maximum point $z_i$ of $P_i$ on $I_i$ satisfies $(P_i)'_t(\xx,z_i)=0$. The point is that this works only if a maximum point is in the interior, for a maximum place located at the endpoint of its interval does not necessarily admit a vanishing derivative.

The central role of this nonsingularity property was elegantly analysed by Shi, who proved by an innovative linear programming approach in \cite{Shi}, that in a rather general context once one has the nonsingularity property (and a much easier technical assumption on so-called properness, see later), all statements, including the strongest intertwining property, follow directly.

Kilgore in his later works formulated several results on various spaces, stating for them the validity of the Bernstein and Erd\H os Conjectures. In particular, the respective claim for the exponentially weighted interpolation on the halfline was stated as Theorem 1 in \cite{K-AMH}. In these papers Kilgore standardized the argumentation for the classical case, always attacking via the nonsingularity statement, which can now be called ``the standard method'', in particular in view of the analysis of Shi, too.

However, it was found in \cite{SzAMH} that the proof arguments relied on the assumption that the maxima are always interior to their intervals, whereas counterexamples demonstrate that this may well fail. Interval maxima may get to the endpoint, and in the case $z_{n+1}=x_{n}$, it may occur that $(P_{n+1})'_t(x_n)<0$; moreover, a full column of the above derivative matrix \eqref{Ak} may become just ${\bf 0}$. In \cite{SzAMH} counterexamples were described to the nonsingularity and to the intertwining properties, too, so that clearly the standard proof cannot be repaired.

Thus for exponentially weighted interpolation we face the fact that the known method of proof, relying on nonsingularity of the said derivative matrices, cannot go through. Nevertheless, we prove here the Bernstein and Erd\H os Conjectures in this setup. Postponing technicalities, we point to the forthcoming Theorem \ref{th:minimax} and Corollary \ref{cor:sandwich}, where exactly these assertions are formulated.

However, the paper contains other findings, too, which might be as crucial for other generalized situations as nonsingularity is where it applies. Note, e.g., that de Boor and Pinkus considered the main result of their paper \cite{CBoorPinkus} their Theorem 2, which will be brought over to our situation in Theorem \ref{t:homeom}. When nonsingularity is not available, this serves as a ``next best thing'', and is also analysed already by Shi in \cite{Shi}. Our approach via this ``homeomorphism property of the difference mapping'' was inspired by successful application of the approach in other contexts, see for example, \cite{TLMS2018} and \cite{Homeo}. We are certain that it will play a crucial role if further settings will be treated.

We should record here that our approach to the equioscillation and minimax questions is also rooted in another, independent and nice discovery of Fenton, who formulated for sums of translates of concave kernels a general equioscillation theorem \cite{Fenton}. Developing his idea with the work of Kilgore \cite{Kilgore} and de Boor-Pinkus \cite{CBoorPinkus} in mind, we could prove very general minimax and equioscillation results in contexts where nonsingularity was not available \cite{Ural, Sbornik, JMAA}. In the case of periodic functions, there is still another predecessor of the approach, due to Kendall, Hardin and Saff \cite{HardinKendallSaff}, and further developed in \cite{TLMS2018}. However, in all these works nonsingularity could be circumvented because we had other strong structural properties available. One was concavity of the appearing kernels, and it was also important that the summands in the sum corresponding to the $|h_k(\xx,t)|$ in the Lebesgue function \eqref{Lebesguefunction} here, all were univariate, depending on just the difference $t-x_k$ with the respective $x_k$. This allowed to apply certain univariate techniques, which of course cannot be used here.

This is the first proof of the conjectures of Bernstein and Erd\H os in a setting where nonsingularity is not available. Therefore, we hope that the methods are capable of extending to other situations, where the same obstacle -- possible singularity, i.e., attaining the interval maxima at endpoints -- excludes the standard approach. This motivates us in formulating several auxiliary results in a more general setting. This was already started by Kilgore \cite{Kilgore1985}, who in particular proved an elegant proposition about linear independence of certain functions with interlacing zeros. We will make good use of this elegant proposition as well as the Markov-type inheritance theorems of Milev and Naidenov \cite{MN} for exponentially weighted polynomials.

The possibility of generalizing to further settings depends on various ingredients including notably an appropriate Markov inheritance lemma, Chebyshev-Haar property of systems of base polynomials as well as the system of their derivatives, and of course the successful implementations of the various steps made here. Therefore, further extensions are not at all obvious, and need to be postponed to other work.

The structure of the paper is as follows. Section \ref{sec:ECHS} is devoted to Extended Chebyshev–Haar systems. We recall some general lemmas and present some extensions of them. The section concludes with a fine description of the sets $X$ and $X^c$ of regular resp. irregular node systems (see Definition \ref{def:X}), which will be found crucial for the later analysis. In Section \ref{sec:zeros} we locate the zeros of the various polynomials $P_i$ from \eqref{Pi}, and also of their derivatives $(P_i)'_t$, and then prove interlacing of zeros of these derivatives. In Section \ref{sec:maximum} we discuss location, uniqueness and continuity of the maximum points $z_i:=z_i(\xx)$ of $P_i$ and also their counterpart, the critical points $w_i:=w_i(\xx)$, which may be outside $I_i$. We prove that the maxima $z_i(\xx)$ ($i=1,\ldots,n+1$) exist uniquely, and, moreover, they are Lipschitz continuous functions of the node system $\xx$. Then we derive that the local maxima functions $m_i(\xx)$ are continuously differentiable, and  calculate the value and the sign of their partial derivatives in several cases. Section \ref{sec:linearindependence} establishes linear independence of the rows of the derivative matrices \eqref{Ak} at least on the set $X$ of regular nodes. However, $X$ does not contain all the node systems and this will be seen to mean the appearance of singularity, too. Section \ref{sec:connected} proves connectedness of the set $X$, and Section \ref{sec:proper} shows that the difference mapping\footnote{For the precise definition see \eqref{gamma}.} $\Gamma$ is proper. The main result on the Bernstein equioscillation theorem follows only in Section \ref{sec:mnimax} as Theorem \ref{th:minimax}. The homeomorphism theorem for $\Gamma$ is contained in Section \ref{sec:homeomorphism}, and in Section \ref{sec:majorization} we derive the Erd\H os sandwich property and more.

\section{General lemmas for Extended Chebyshev-Haar Systems}\label{sec:ECHS}

A vector space of functions $\WW:=\Wn$ is a rank $n+1$ Extended Chebyshev-Haar System -- to be shortened as ECHS henceforth --, if its dimension is $n+1$, and the only function $f$ from the family and having $n+1$ zeros -- counted according to multiplicities -- is the identically zero function $\Null$. In this part of the preliminary analysis we will consider an almost arbitrary ECHS, which, however, will be assumed to have suitable differentiability properties, as need be.

Most of the time we will assume $\Wn \subset C^2(\II)$ and will count multiplicities of zeros up until 3: so for $f \in \Wn$ and $z$ a zero of it the multiplicity of the zero is 1 if $f'(z) \ne 0$, is 2 if $f(z)=f'(z)=0$, but $f''(z) \ne 0$, and is 3 if $f(z)=f'(z)=f''(z)=0$. Our notion of ECHS will refer to this zero multiplicity count -- for a similar convention see, e.g., \cite{MN}. (In general, when we do not want to assume higher order differentiability of the functions of our system, then we can neither interpret higher order multiplicities.)

Our model study in this direction is the work of Kilgore and Cheney \cite{Kilgore-Cheney}. However, our main concern here will be the removal from the set of conditions the standing assumption of \cite{Kilgore-Cheney} that $\One \in \WW$. As is exposed in \cite{SzAMH}, this innocent-looking assumption is a very crucial one, lacking of which causing collapse of the ``standard approach'' via nonsingularity of \eqref{Ak} to the Bernstein and Erd\H os Conjectures. So, here our aim is to save what can be saved from the analysis of \cite{Kilgore-Cheney}, and to generalise, extend and push further them towards the particular needs in our approach. The results obtained here will be essentially used throughout, from results on derivatives of the interval maxima functions $m_i$ in Section \ref{sec:maximum} to connectedness of the ``regular set'' $X$ and continuity of the ``separator function'' $\phi$ in Section \ref{sec:connected}.

We will always consider the maximum norm of functions. Note that if the weights are ``built in'' the functions of the ECHS, then weighted polynomial interpolation (with the second interpretation of it) is a particular case. Indeed, for any node system $\xx$ there exist unique interpolatory basis functions $h_k$ with $h_k(x_j)=\delta_{j,k}$, the Lebesgue function can be constructed as discussed, and the ``generalised polynomials'' $P_i$, as well as the signs $\ve_{k,i}$ are defined the same way.

The below lemmas generalise some known nice formulas. According to our study, these can be traced back to the work of Kilgore and Cheney \cite{Kilgore-Cheney}. However, Kilgore and Cheney use that ${\bf 1}$ belongs to the given ECHS, which we want to avoid. We give below in Lemma \ref{l:magic} the properly adjusted version. These lemmas will have multiple use in our work. One direction culminates in Corollary \ref{cor:magic}. It is mentioned without proof in \cite{K-AMH}, where it is called only ``an observation''. However downplayed in \cite{K-AMH}, we found it a miraculous formula, which we will use extensively.

To start with, we generalise Lemma 2 from \cite{Kilgore-Cheney}. Kilgore and Cheney showed that if $\xx, \xx'\in S$ are given by $x'_k=x_k$ for each $k\in \{1,\ldots,n\}\backslash\{j\}$ and $x_{j-1} < x_j' < x_j$, then
\begin{align*}\label{K-Ch_diff}
P_j(\xx,t)-P_j(\xx',t)&=(P_j(\xx,x_j')-1)h_j(\xx',t),
\\
P_{j-1}(\xx,t)-P_{j-1}(\xx',t)&=(P_{j-1}(\xx,x_j')-1)h_j(\xx',t).
\end{align*}
Here we show that a similar formula holds true for any $P_i$ even if $i$ is not necessarily equal to $j$ or $j-1$. It is of importance that $\One \in \WW$ is not postulated for this.

\begin{lemma}\label{l:magic-I}
Let $\Wn \subset C^1(\II)$ be an ECHS of continuously differentiable functions on $\II$. Let $\xx, \xx'\in S$ be given with $x'_k=x_k$ for each $k\in \{1,\ldots,n\}$ with $k \ne j$, and satisfying $x_{j-1} < x_j' < x_j$.
Then, for every $1 \le j\le n$ and $1\le i\le n+1$ we have
\begin{equation}\label{K-Ch_diff_j}
P_i(\xx,t)-P_i(\xx',t)=(P_i(\xx,x_j')-\varepsilon_{j,i})h_j(\xx',t).
\end{equation}

\end{lemma}
\begin{proof} Obviously, $P_i(\xx,t)-P_i(\xx',t)$ is an element of the ECHS $\WW_{n}$ and it takes 0 at $n$ points (at the nodes $x_k$, $k=0,\ldots,n$, $k\not=j$). Since $h_j(\xx',t)$ has zeroes at the same points, the  Chebyshev-Haar property of $\WW_{n}$ guarantees, that there is a constant $c$ such that
\[
P_i(\xx,t)-P_i(\xx',t)=c\cdot h_j(\xx',t).
\]
Moreover, for the proper constant factor we have that
\[
c=\restr{\frac{P_i(\xx,t)-P_i(\xx',t)}{h_j(\xx',t)}}{t=x'_j}=P_i(\xx,x_j')-\varepsilon_{j,i},
\]
furnishing \eqref{K-Ch_diff_j}.
\end{proof}

% Namely, we are going to prove that

\begin{lemma}\label{l:magic}
Let $\Wn \subset C^1(\II)$ be an ECHS of continuously differentiable functions on $\II$. Let $\xx \in S$. Then, for every $1 \le j\le n$ and $1\le i\le n+1$ we have
\begin{equation}\label{magic-middleform}
\frac{\partial P_i}{\partial x_j}(\xx,t)= - h_j(\xx,t) (P_i)'_t(\xx,x_j) \qquad (t \in \RR).
\end{equation}
\end{lemma}

\begin{proof} Taking limits $x_j'\to x_j-$ in \eqref{K-Ch_diff_j} we find
\begin{align*}\label{magic-middleform}\notag
-\frac{\partial P_i}{\partial x_j}(\xx,t)&=\lim\limits_{x_j'\to x_j} \frac{P_i(\xx,t)-P_i(\xx',t)}{x_j'-x_j}=\lim\limits_{x_j'\to x_j} \frac{(P_i(\xx,x_j')-\varepsilon_{j,i})h_j(\xx',t)}{x_j'-x_j}
\\&
=\lim\limits_{x_j'\to x_j} \frac{P_i(\xx,x_j')-P_i(\xx,x_j)}{x_j'-x_j}\cdot h_j(\xx,t) =(P_i)'_t(\xx, x_j)\cdot h_j(\xx,t),
\end{align*}
proving \eqref{magic-middleform}.
\end{proof}

\begin{corollary}\label{l:magic-derivalt}
Let $\Wn \subset C^2(\II)$ be an ECHS of twice continuously differentiable functions on $\II$. Let $\xx \in S$. Then, for every $1 \le j\le n$ and $1\le i\le n+1$ we have
\begin{equation}\label{magic-derivative}
\frac{\partial (P_i)'_t}{\partial x_j}(\xx,t)= - (h_j)'_t(\xx,t) (P_i)'_t(\xx,x_j) \qquad (t \in \RR).
\end{equation}
\end{corollary}

For the next argument we need the description of the zeros of the $n+1$st interval Lebesgue function, $P_{n+1}(\xx,t)$ as well as the description of the zeros of its derivative $(P_{n+1})'_t(\xx,t)$.

Consider the ECHS $\WW$ and a fixed basis $\ff_0,\ldots,\ff_n$ of it. Each $h_j(\xx,\cdot)$ is a linear combination $\sum_{i=0}^n \alpha_i^j \ff_i=h_j(\xx,\cdot)$ of the fixed basis functions, moreover, the respective coefficients from the solution vector $\ba^j=(\alpha_0^j,\ldots,\alpha_n^j)^T$ can be obtained from the inhomogeneous linear equation
\[
\begin{bmatrix} \ff_0(x_0) & \ldots & \ff_n(x_0)
\\ \vdots & \ddots & \vdots
\\ \ff_0(x_i) & \ldots & \ff_n(x_i)
\\ \vdots & \ddots & \vdots
\\ \ff_0(x_n) & \ldots & \ff_n(x_n)
\end{bmatrix}
 \cdot \ba^j = \ee_j := \begin{bmatrix} 0 \\ \vdots \\ 0 \\ 1 \\ 0 \\ \vdots \\0 \end{bmatrix}
\]
Since $\WW$ is an ECHS, and $\{\ff_i\}_{i=0}^n$ is a basis, the alternating determinant of the left-hand side matrix does never vanish for $\xx \in S$. It follows that the system has a unique solution $\ba^j$, expressed by Cramer's rule, and therefore depending continuously on $\xx$. That is, the functions $h_j( \xx,\cdot)$ depend continuously on $\xx \in S$. Moreover, the derivatives $(h_j)_t'(\xx,\cdot)=\sum_{i=0}^n \alpha_i^j(\xx) \ff'_i$ depend continuously on $\xx \in S$, too. So we see that $P_i(\xx,t)$ and also $(P_i)'_t(\xx,t)$ depend continuously on $\xx \in S$.

Here we will need a preliminary analysis of the location of zeros in particular of $P_{n+1}$ and $P_{n+1}'$. $P_{n+1}$ attains $\pm 1$ alternatively on $x_i$, taking $+1$ at $x_n$, and then it attains $0$ also at $x_{n+1}=b=\infty$. So, it has $n$ zeroes $y_i \in (x_{i-1},x_i)$ for $i=1,\ldots,n$, and by the ECHS property all these are simple and there can be no more. Therefore, writing $y_{n+1}=\infty$, the derivative $P_{n+1}'$ has zeroes $w_i \in (y_i,y_{i+1})$ for $i=1,\ldots,n$. Assuming here that also the derivative system $\WW'$ of functions from $\WW$ is an ECHS, too this means that also the derivative has no more zeroes, and all these are simple.

Actually, here we can refer to Milev and Naidenov \cite{MN}, who introduced the following.

\begin{definition}\label{def:prop-P} An ECHS $\WW$ satisfies Property (P) if there exist numbers $\de_0, \de_n \in \{0,1\}$ such that, for every oscillating generalized polynomial $f$ of the system, its derivative $f'$ has no zeros other than the $n-1$ zeros guaranteed by Rolle's theorem between consecutive roots of $f$, together with exactly $\de_0$ zero in the left outer interval and exactly $\de_n$ zero in the right outer interval.
\end{definition}

So, instead of the assumption that $\WW'$ is an ECHS, we can as well say only that $\WW$ has Property (P) with $\delta_0=0$ and $\delta_n=1$. Both versions are well-known and easy for the system of exponentially weighted polynomials.

\begin{definition}[Regular set]\label{def:X} We define the \emph{regular set} $X\subset S$ as
\[
X:=\{ \xx \in S : m_{n+1}(\xx) >1 \}
\]
\end{definition}

It is clear from the above that $m_{n+1}(\xx)$ is a continuous function of $\xx \in S$, hence the level set $X$ is open. Consequently, $X^c:=S \setminus X = \{ \xx \in S : m_{n+1}(\xx) =1 \}$ is (relatively) closed in $S$.

As the interval $(y_{n},\infty)$ contains only one and simple zero $w$ of the derivative of $P'_{n+1}$, whereas $P_{n+1}$ vanishes at the endpoints and is positive in the interval, we necessarily have that $P_{n+1}$ is strictly increasing on $[y_n,w]$ (from 0 to $P_{n+1}(\xx,w)$), and is strictly decreasing on $[w,\infty)$ (from $P_{n+1}(\xx,w)$ to 0). In other words, $P_{n+1}$ is always strictly quasi-concave on $(y_n,\infty)$. So, $P_{n+1}(\xx,w)$ is a strict global maximum point, with its value $P_{n+1}(\xx,w)\ge P_{n+1}(\xx,x_n)=1$. This also means that on $I_{n+1}(\xx)=[x_n,\infty)\subset [y_n,\infty]$, too, there is a unique strict maximum point $z=z_{n+1}(\xx)$, which is equal to $w$ if $w\ge x_n$, and is otherwise the left endpoint $x_n$, for then $P_{n+1}$ is decreasing on $I_{n+1}$.

So if $w > x_n$, then $m_{n+1}(\xx)=P_{n+1}(\xx,w)> P_{n+1}(\xx,x_n)=1$, as $P_{n+1}$ is strictly increasing on $[x_n,w]$; but if $w\le x_n$, then $z=x_n$ and $m_{n+1}(\xx)=P_{n+1}(\xx,x_n)=1$. So we get the following.

\begin{proposition}\label{prop:equiv} Let $\WW \subset C^1(\II)$ be an ECHS satisfying Property (P) with $\de_n=1$ (and $\de_0 \in \{0,1\}$ arbitrary).
%where $\de_0=0$ and $\de_n=1$,
Then the following are equivalent.
\begin{enumerate}
\item $\xx \in X$;
\item The unique zero $w$ of $(P_{n+1})'_t(\xx,\cdot)$ on $(y_n,\infty)$ satisfies $w>x_n$;
\item The unique maximum point $z=z_{n+1}(\xx)$ of $P_{n+1}(\xx,\cdot)$ on $I_{n+1}$ lies in $\intt I_{n+1}$;
\item $z_{n+1}=w>x_n$;
\item $(P_{n+1})'_t(\xx,x_n)> 0$.
\end{enumerate}
\end{proposition}

\begin{lemma}\label{l:xstar} Assume that $\WW \subset C^1(\II)$ is an ECHS satisfying Property (P) with $\de_n=1$ (and $\de_0 \in \{0,1\}$ arbitrary). Let $\wx:=(x_1,\ldots,x_{n-1}) \in S^{n-1}:=\{\wx \in \RR^{n-1} : 0<x_1<\ldots<x_{n-1} \}$ be fixed.

Then there exists one, and only one value $x_n^* \in (x_{n-1},\infty)$ such that for $\xx^*:=(\wx,x_n^*)$ it holds $w=x_n^*$, with $w$ the unique derivative zero of $(P_{n+1})'_t(\xx^*,\cdot)$ in $(y_n,\infty)$.

In other words, there is a unique $x_n^*$ with the property that $(P_{n+1})'_t(\xx^*,x_n^*)=0$.

Moreover, $\xx^* \in \partial X$, for $x_n>x_n^*$ we have $\xx:=(\wx,x_n) \in \intt X^c$ with $(P_{n+1})'_t(\xx,x_n)<0$ and for $x_n<x_n^*$ we have $\xx \in X$ and $(P_{n+1})'_t(\xx,x_n)>0$.
\end{lemma}
\begin{proof} First we observe that there exist $x', x'' \in J:=(x_{n-1},\infty)$ such that $\xx' \in X, \xx'' \in X^c$.

Let us choose any node $\xi \in J$, and consider the interpolating polynomial $Q$ which assumes $(-1)^{n-i}$ at $x_i$ for $i=0,1,\ldots,n-1$, and $2$ at $\xi$. That polynomial exists uniquely, as $Q=\sum_{i=0}^n b_i h_i((\wx,\xi),\cdot)$ with $b_i=\ve_{i,n+1}$  for $i=0,\ldots,n-1$ and $b_{n+1}=2$. As $Q$ assumes $-1$ at $x_{n-1}$, $+2$ at $\xi$ and tends to $0$ at $\infty$, there are two points $x_{n-1}<x'<\xi<x''$, where $Q(x')=Q(x'')=1$. Choosing any of them, we get node systems $\xx', \xx''$ satisfying $P_{n+1}(\xx',x_i)=Q(x_i)=P_{n+1}(\xx'',x_i)$ for $i=0,\ldots,n-1$, and also matching on $x_n'=x'$ resp. $x_n''=x''$ by construction. In view of uniquness of the interpolating polynomials, we find $P_{n+1}(\xx',\cdot)\equiv Q\equiv P_{n+1}(\xx'',\cdot)$. However, the maximum point $w$ of $Q=P_{n+1}(\xx',\cdot)=P_{n+1}(\xx'',\cdot)$ is somewhere between $x'$ and $x''$ in view of the fact that any $P_{n+1}$ is quasi-concave and $Q(\xi)=2>1$. It follows that for $\xx'$ we have $w>x_n'=x'$ and $\xx' \in X$, whereas for $x''$ we have $w < x_n''=x''$ and $\xx'' \in X^c$.

Moreover, we also see $(P_{n+1})'_t(\xx',x_n')=Q'(x')>0$ and $(P_{n+1})'_t(\xx'',x_n'')=Q'(x'')<0$. As $(P_{n+1})'_t(\xx,x_n)$ changes continuously in function of the nodes, the same will also hold in a neighborhood of the nodes $\xx'$, resp. $\xx''$. Therefore, in view of Proposition \ref{prop:equiv}
even a small neighborhood of these points must belong to $X$, resp. $X^c$.

This also furnishes existence of a point $x_n^*$ on the boundary of $\{x_n : \xx \in X\}$ -- indeed, such a value must occur somewhere after $x'$, but not later than $x''$. As $X$ is open, any boundary point $\xx^*=(\wx,x_n^*)$ belongs to $X^c$ (but it is also on the boundary of $X^c$). Moreover, as in $X$ $(P_{n+1})'_t(\xx,x_n)<0$, and on $X^c$ it is $\le 0$, on their common boundary the only possibility is $(P_{n+1})'_t(\xx^*,x_n^*)=0$.

In the following we aim at describing better the point set $\{x_n : \xx \in X\}$, eventually deriving uniqueness of $\xx^*$. So let us take any $x_n$ with $\xx \in X$. We claim that then for $x_n' \in (x_{n-1},x_n)$ we also have $\xx':=(\wx,x') \in X$.

Let us recall the formula
\eqref{K-Ch_diff_j} for the case $j=n$ and $i=n+1$ here:
\begin{equation}\label{K-Ch_diffnnplusone}
P_{n+1}(\xx,t)-P_{n+1}(\xx',t)=(P_{n+1}(\xx,x_n')-1)h_n(\xx',t).
\end{equation}
As $\xx \in X$, we have $P_{n+1}(\xx,\cdot)$ strictly increasing on $[y_n,w] \supset [y_n,x_n]$, hence $P_{n+1}(\xx,x_n')<P_{n+1}(\xx,x_n)=1$ whenever $y_n\le x_n'\le x_n$. Further, the same is trivial for $x_n' \in (x_{n-1},y_n)$, because on that interval $P_{n+1}(\xx,\cdot) <0$, given that it is $-1$ at $x_{n-1}$, 0 at $y_n$, and $P_{n+1}(\xx,\cdot)$ has no more zeros on that interval.

So finally we derive that the right hand side of \eqref{K-Ch_diffnnplusone} is negative on $[x_n,\infty)$. Indeed, we have just seen that the first factor is negative, whereas the last term $h_n(\xx',t)$ is positive on $[x'_n,\infty)$ containing $[x_n,\infty)$. This means that also the left hand side is negative, that is, we find
\[
P_{n+1}(\xx,t) < P_{n+1}(\xx',t) \qquad (t \in [x'_n,\infty) ).
\]
It follows $m_{n+1}(\xx)=\max_{[x_n,\infty)} P_{n+1}(\xx,t) \le \max_{[x_n,\infty)} P_{n+1}(\xx',t)\le m_{n+1}(\xx')$, and $\xx' \in X$, as claimed. Moreover, here we have the strict inequality $\max_{[x_n,\infty)} P_{n+1}(\xx,t)< \max_{[x_n,\infty)} P_{n+1}(\xx',t)$, too, for $P_{n+1}(\xx,\cdot)$ and $P_{n+1}(\xx',\cdot)$ are continuous and positive functions on $[x_n,\infty)$, whereas their limit at $\infty$ is 0.

So we derive $m_{n+1}(\xx)<m_{n+1}(\xx')$, and thus the latter exceeds 1, hence $\xx' \in X$, as wanted. This in particular means also that $\{x_n : \xx \in X\}$ is an open, finite interval with left endpoint $x_{n-1}$.

To estimate with a strict inequality was not important for the case $\xx \in X$, as then $m_{n+1}(\xx')\ge m_{n+1}(\xx)>1$ already suffices. However, note that the above argument works mutatis mutandis even for any $\xx \in X^c$ satisfying $(P_{n+1})'_t(\xx,x_n)=0$. Then we have $w=x_n$, but the only place where $w$ had a role was using $[y_n,w] \supset [y_n,x_n]$, which still holds true.

This furnishes a stronger result: if $\xx \in X$ or $\xx \in X^c$ but with $(P_{n+1})'_t(\xx,x_n)=0$, then we always have $\xx' \in X$ for any $\xx'=(\wx,x_n')$ with $x_n'<x_n$ (but of course satisfying $x'_n>x_{n-1}$).

Now we can draw an important conclusion: there can be only one value $x_n^*$ with $(P_{n+1})'_t(\xx,x_n^*)=0$. Indeed, if there were two, the one with smaller last coordinate should have been in $X$, a contradiction in view of (5) of Proposition \ref{prop:equiv}. So, $x_n^*$ is unique.

However, as remarked earlier, any point $\xx$ with $(P_{n+1})'_t(\xx,x_n)<0$ is in the interior of $X^c$, so that on the halfline $(x_{n-1},\infty)$ there can only be one and only one value $x_n^*$ with $(\wx,x_n^*)$ a boundary point of $X^c$, whereas all other values $x_n >x_n^*$ give rise to $\xx=(\wx,x_n) \in \intt X^c$.
\end{proof}

Here we arrive at a crucial finding, which we will need in Section \ref{sec:connected}.

\begin{proposition}\label{prop:fat} There exists a continuous function $\phi : S^{n-1} \to \RR$, such that $\phi(x_1,\ldots,x_{n-1})>x_{n-1}$ and $X^c$ is the closed epigraph of $\phi$. Furthermore, $X^c$ is \emph{fat} in $S$, i.e., $X^c = \overline{\intt X^c}$.
\end{proposition}
\begin{proof} We have already found $\phi(x_1,\ldots,x_{n-1}):=x_n^*$ with the epigraph property. Observe that $X^c = \overline{\intt X^c}$ follows directly from the last assertion of Lemma \ref{l:xstar}.

Next, we prove that $\phi$ is bounded locally, say in a $\delta$-neighborhood of $\wx=(x_1,\ldots,x_{n-1})$, with a small enough $\de>0$. We choose $0<\de<1$ such that $\de<\min_{i=1,\ldots,n-1} x_i-x_{i-1}$, and fix a value $\xi>x_{n-1}+1$. If we repeat the interpolation construction of the beginning of the proof of Lemma~\ref{l:xstar} for all points $\wy \in B(\wx,\de)$ and with the fixed $\xi>y_{n-1}$, the interpolation polynomial $Q((\wy,\xi);\cdot)$ --and the point $(\wy,y'')\in X^c$ satisfying $Q((\wy,\xi),y'')=1$ and $Q'_t((\wy,\xi),y'')<0$-- is obtained. As discussed in the mentioned proof, the value $\phi(\wy)=y_n^*<y''$, always. However, $y''$ is the largest root of the polynomial equation $Q((\wy,\xi);\cdot)=1$, where $Q$ is the interpolation polynomial with interpolation constant values $(b_i)$ and interpolation nodes $(\wy,\xi)$, so that this root depends on the varying nodes $\wy$ continuously, and therefore remains bounded in $B(\wx,\de)$. It follows that also $y_n^*<y''$ is bounded, as wanted.

Continuity follows easily, too. Indeed, let $\xx =(x_1,\ldots,x_{n-1},x_n^*) \in \partial X$, meaning that $(P_{n+1})'_t(\xx, x_n^*)=0$. Take a sequence of values $\wxk=(x_1^k,\ldots,x_{n-1}^k) \to \wx=(x_1,\ldots,x_{n-1})$. Then by construction $(P_{n+1})'_t((\wxk,\phi(\wxk)),\phi(\wxk))=0$ for all $k$.

The sequence $\phi(\wxk)$ is bounded, hence it admits
a convergent subsequence. Let $\phi(\wx_{k_\ell})\to y$. Passing to the limit and using the continuity of $(P_{n+1})'_t$ in all variables we obtain $$(P_{n+1})'_t((\wx,y),y)=0.$$ The uniqueness of $x_n^*$ entails that only $y=x_n^*$ is possible. Thus every convergent subsequence of $\{\phi(\wxk)\}$ has the same limit $x_n^*$. Since the sequence is bounded, this implies $\phi(\wxk)\to x_n^*$. Therefore, $\phi$ is continuous.
\end{proof}

\section{Location of zeros of $P_i$ and their derivatives}\label{sec:zeros}

\subsection{Roots of the $P_i$.} In the first part of our work, we are going to follow the key steps that were applied in \cite{Kilgore, K-AMH, CBoorPinkus}. Here we perform a general analysis, with a system $\WW_n=w\Pn$ of weighted polynomials forming an ECHS. At some points, restrictions do apply, so that we signal the necessary assumptions in full.

To follow that standard approach, first we need the location of the zeros of the functions $P_i$ ($i=1,\ldots,n+1$) and then the ones of their derivatives, too. To explore this, we will occasionally consider the corresponding polynomial functions,
\begin{equation}\label{pi}
p_i(\xx,t):=p_i(t):=\frac{P_i(t)}{w(t)}=\sum_{k=0}^n \frac{\vki}{w(x_k)} \ell_k(t), \qquad i=1,\ldots,n+1,
\end{equation}
too. These will be called \emph{the polynomial parts} of the functions $P_i$. The zeros of the polynomials $p_i$ ($i=1,\ldots,n+1$) coincide the ones of $P_i$. In line of this, speaking loosely, we will talk about $\deg P_i$, meaning of $\deg p_i$, and will refer to the known fact that $p_i$, hence $P_i$, can have at most as many zeroes as its degree. At the $n+1$ nodes $x_0:=0, x_1, \ldots, x_n$ the function values of $P_i$ alternate signs except for $x_{i-1}, x_{i}$, where the signs are both positive. More precisely, these sign equalities hold true for $i=1,\ldots,n$; thus, for such $i$, the function $P_i$ exhibits $n$ alternating signs and consequently has $n-1$ zeros between the alternating nodes. Further, these zeros are unique and simple in their intervals, for otherwise to meet the required signs at the endpoint nodes, there would have to be at least three zeroes (counted with multiplicities), altogether the number of zeros climbing to at least $n-1+2=n+1$ exceeding the degree of $P_i$. As already told, $P_i(x_{i-1})=P_i(x_i)=1$ ($i=1,\ldots,n$) and hence for similar reasons as before, there is no zero in $I_i$.

Recall that we have already seen that $P_{n+1}$ has $n$ different, simple zeroes, one in each $(x_i,x_{i+1})$, $i=0,\ldots,n-1$, it cannot have more, and its degree is exactly $n$, not less.

Let us denote the leading coefficient\footnote{More precisely, the coefficient of the $n$th power term in the polynomial part $p_i$ belonging to $P_i$.} of $P_i$ as $a_i$. Also, denote the leading coefficients of the Lagrange fundamental polynomials $\lk$ as $b_k$. Then
\begin{equation}\label{bk}
b_k=\prod_{\substack{j=0\\j\ne k}}^n \frac{1}{x_k-x_j}=(-1)^{n-k} |b_k|, \quad k=0,\dots, n.
\end{equation}
To compute $a_i$ is easy from the formulas \eqref{eq:hkt}, \eqref{epsilonki} and \eqref{Pi}:
\begin{equation}\label{ai}
\begin{aligned}
a_i= \sum_{k=0}^n \frac{\vki}{w(x_k)} b_k &= \sum_{k=0}^n \frac{\vki}{w(x_k)} (-1)^{n-k} |b_k|
%%%\qquad
%%%\\ &
= (-1)^{n+1+i} \left\{\sum_{k=0}^{i-1} \frac{|b_k|}{w(x_k)} - \sum_{k=i}^{n} \frac{|b_k|}{w(x_k)} \right\}.
\end{aligned}
\end{equation}

From this it is clear that $a_{n+1}>\cdots > (-1)^{n+1+i}a_i > \cdots > (-1)^{n+2} a_1$, and so $a_{n+1}$ is positive, and there exists a certain index $r \le n$ such that $(-1)^{n+1+s} a_s \le 0$ for $s\le r$ (and is strictly $<0$ for $s<r$), whereas $(-1)^{n+1+q} a_q>0$ for $q>r$.

From now on, we can follow the arguments appeared in \cite{SzAMH} with slight modifications. Therefore, we do not present every calculation in a precise form, but we emphasize the differences and list the main properties of the functions $P_i$.

Above we found $n$ simple zeros of $P_{n+1}$, and $n-1$ for the further functions $P_i$ for each $i=1,\ldots n$, one in each $\intt I_k$, $k=1,\ldots, n$, save the very $I_i$. If $a_r=0$, then $P_r$ cannot have more zeros--we found them all, (as it cannot be identically zero given that $P_r(x_r)=1$).

For the remaining functions $P_i$, we have to  determine the sign of function values at $x_0$, $x_n$ and then also for very large and very small\footnote{Formally, there is a problem here if the weight $w(t)$ is not defined for $t<0$. However, the root analysis goes over to the polynomials \eqref{pi}, if we do not have $w(t)$ on the negative real axis. (Note that leading coefficients etc. are always defined for $p_i$.) Given that now we focus on the exponential weight, we save the reader from translating back and forth between $P_i$ and $p_i$.} $\pm T$, respectively, from the leading coefficient. Indeed, $\lim_{t \to \pm \infty} P_i(t)/w(t) t^n$ equals to $a_i$ and $(-1)^{n} a_i$, respectively, hence for large, resp. small, values $\pm T$ we have
$$
\sign P_i ( T) = \sign a_i \quad \text{ and } \quad \sign P_i ( -T) = (-1)^{n} \sign a_i.
$$
So, let $i$ be any index between $1$ and $n$, excluding $i=r$ if $a_r=0$. For these indices $a_i\ne 0$ and we have $P_i(x_0)=\ve_{0,i}=(-1)^{1-i}$ and $P_i(x_n)=\ve_{n,i}= (-1)^{n-i}$. Therefore, if $(-1)^{n+1+i} a_i <0$ (i.e., if $i<r$ and even for $i=r$ if $a_r \ne 0$), then $\sign P_i(-T)= (-1)^{n+n+1+i-1}=(-1)^i$, whereas $P_i(x_0)=(-1)^{1-i}$, and there is one more sign change in $(-T,x_0)$, pointing to one more zero of $P_i$. Similarly, if $i>r$ -- and then automatically $(-1)^{n+1+i}a_i>0$ --, then $P_i(x_n)=(-1)^{n-i}$ and $\sign P_i(T)= \sign a_i= (-1)^{n+1+i}$, and there is one more sign change, hence another zero, in $(x_n,T)$.

Altogether, we obtain that for $i$ from $1$ up to $r-1$ the function $P_i$ has an extra zero in $I_0:=(-\infty,x_0)$, for $i=r$ either we have the same or if $a_r=0$, then there is no extra zero outside of $(x_0,x_n)$, and for $r< i \le n$ $P_i$ has one extra zero in the interior $(x_n,\infty)$ of $I_{n+1}$.

\bigskip
According to the above we now define the index set $\Ji$ to contain all those indices $j$ for which there is a root of $P_i$ in $I_j$, where $j=0, 1,\ldots,n+1$. Then we will have
\begin{equation}\label{Jistardef}
\Ji=\begin{cases}
 \{0,1,\ldots,n\}\setminus \{i\} & \textrm{if} \quad 1\le i<r
\\  & \textrm{and if} \quad i=r \quad \textrm{and} \quad {a_r} \ne 0
\\ \{1,\ldots,n\}\setminus \{i\} & \textrm{if} \quad i=r \quad \textrm{and} \quad {a_r} = 0
\\ \{1,\ldots,n,n+1\}\setminus \{i\} & \textrm{if} \quad r<i\le n+1
\end{cases}.
\end{equation}

\subsection{Interlacing of roots of $P_i$} This part of the analysis goes mutatis mutandis of de Boor and Pinkus \cite{CBoorPinkus}, see Lemma 1 there.

To deal with zeros of different $P_i$, let $y_k^{(i)}$ denote the unique zero of $P_i$ in $I_k$; we have such a zero precisely for $k \in \Ji$. Further, if $i<j$ we denote $\Kij:=\Ji \cap \Jj \setminus \{i+1,\ldots,j-1\}$.

\begin{lemma}\label{l:zerosinsame} Let $1 \le i < j \le n+1$ be arbitrary. If $k \in \Kij$, then the zeros of $P_i$ and $P_j$ satisfy $y_k^{(j)} < y_k^{(i)}$; and if $i<k<j$, then $y_k^{(i)} < y_k^{(j)}  $.
\end{lemma}
Note that $\{i+1,\ldots,j-1\} \subset \Ji \cap \Jj$ and hence the two index sets in the conditions sum up precisely to $\Ji \cap \Jj$.

\begin{proof}
We will consider two auxiliary functions $G$ and $H$ which were introduced in \cite{SzAMH}, following \cite{CBoorPinkus}. Just like there, the sign changes and zeros of $G$ and $H$ guarantee the desired inequalities for the zeros of functions $P_i$, $(i=1,\dots,n+1)$. Namely, let
\begin{equation}\label{ghdef}
\begin{aligned}
&G(t):= (-1)^{i+1} P_i(t)+(-1)^{j+1} P_j(t) \quad \textrm{and}\\
&H(t):= (-1)^{i+1} P_i(t)-(-1)^{j+1} P_j(t) .
\end{aligned}
\end{equation}

Moreover, the function $H$, that was applied to prove the inequalities, $y_k^{(j)} < y_k^{(i)}$ for $k \in \Kij$, has exactly the same form in both cases of $\WW_n$ and $\YY_n=\WW_n\oplus \bf{1}$, already analysed in \cite{SzAMH}. Therefore, we omit the presentation of the proof of that part of the stated inequalities, and restrict to the case when $i<k<j$.

However, the case of $G$, whose examination is needed to verify the inequalities $y_k^{(i)} < y_k^{(j)}$ if $i<k<j$, slightly differ from the analysis for $\YY_n$. So let now $I_k=[x_{k-1},x_{k}]$ be any of the intervals with $i<k<j$. We know that there are sign changes and hence zeroes of $P_i$ and $P_j$ in $I_k$. Also, $G(x_{k-1})=G(x_k)=0$.

Here we can add, that $G$ has constant sign on $(x_{k-1},x_{k})=\intt I_k$. In fact, from \eqref{Pi} we obtain
\begin{equation}\label{gPiPjIk}
\begin{aligned}
G &=(-1)^{i+1} P_i +(-1)^{j+1} P_j = \sum_{m=0}^n \frac{\left((-1)^{m+\chi_{i\le m}} + (-1)^{m+\chi_{j\le m}}\right)}{w(x_m)} \ell_m \\
&= 2\sum_{m=0}^{i-1} \frac{(-1)^m}{w(x_m)} \ell_m + 2\sum_{m=j}^n \frac{(-1)^{m+1}}{w(x_m)} \ell_m,
\end{aligned}
\end{equation}
and on $\intt I_k$ the function $\ell_m$ has the sign $\ve_{m,k}=(-1)^{m+k+1+\chi_{k \le m}}$, so that all the terms of the expression on the right have the sign $(-1)^{k+1}$. Therefore, $\sign G =(-1)^{k+1}$ on $\intt I_k$.

Consider now $f:=(-1)^{k+1} G=(-1)^{i+k} P_i -(-1)^{j+k+1} P_j$ on $I_k$. It is zero at the endpoints $x_{k-1}, x_{k}$, and positive inside the interval.

Let us put here $\ff:=(-1)^{k+i}P_i$ and $\psi:=(-1)^{k+j+1}P_j$. As $k<j$, the formula \eqref{epsilonki} guarantees that $\psi(x_{k})=(-1)^{k+j+1}P_{j}(x_k)= (-1)^{k+j+1} \vkj =1>0$, therefore the functions $f$, $\ff$, and $\psi$ on $I_k$ satisfy the conditions of the simple lemma implicitly used already in \cite{CBoorPinkus} and spelled out in \cite{SzAMH} as Lemma 2. This Lemma implies that on $\intt I_k$ any zero of $\ff$ precedes the largest zero of $\psi$. However, these functions each have precisely one zero, namely $y^{(i)}_k$ and $y_k^{(j)}$ (the zeros of $P_i$ and $P_j$ in $I_k$), respectively, hence we are led to $y^{(i)}_k<y_k^{(j)}$, as wanted.
\end{proof}

\begin{lemma}\label{l:interlacing} For any $1\le i < j \le n+1$, the zeros of $P_i$ and $P_{j}$ strictly interlace, that is, strictly between any two consecutive (simple) zeros of one of the polynomials, there is exactly one (simple) zero of the other, too.

Moreover, the ordering of the $P_i$ is counter-cyclical: $P_r\prec P_{r-1} \prec \dots\prec P_1 \prec P_{n+1} \prec P_n \prec \dots \prec P_{r-1}$, unless $a_r=0$, in which case $P_r$ is only a degree $n-1$ polynomial, whose $n-1$ zeros are interlacing with the $n$ zeros of any other $P_i$ with $i\ne r$ (and the ordering of the remaining $P_i$ remains counter-cyclical starting with $P_{r-1}$ in this case).
\end{lemma}

\begin{proof} To infer the interlacing and precedence ordering of the root sequences of our polynomials from the zero locations established in Lemma \ref{l:zerosinsame} is a routine derivation, exactly matching that of the classical case, see e.g., \cite{CBoorPinkus} or, for a more detailed deduction, \cite{SzAMH}. Therefore, we save the reader from the details here.
\end{proof}

Finally, we get the ordering of all the roots of all the $P_i$.

\begin{corollary}\label{c:order_roots_total}
The roots $y_k^{(i)}$ ($k=0,\dots,n+1$, $i=1,\dots,n+1$) have the following order:
\begin{align}\label{order_total}
(y_0^{(r)}<)~y_0^{(r-1)}&<\dots<y_0^{(1)}< x_0 <y_1^{(n+1)}<y_1^{(n)}<\dots<y_1^{(2)} \notag
\\ & < x_1 < \dots < x_n <y_{n+1}^{(n)}<\dots<y_{n+1}^{(r+1)},
\end{align}
where $y_0^{(r)}$ exists, if and only if $\deg P_r=n$.
\end{corollary}

\subsection{Interlacing of roots of the derivatives $P_i'$}\label{sec:Markov}

Now, we are in a position to describe the location of the zeroes of the \emph{derivatives} $P_i'$ of the functions $P_i$. To do this we need an appropriate generalization of a famous lemma of V. A. Markov \cite{Markov} on the interlacing of the roots of the derivatives of polynomial functions with interlacing roots. In the 2000's there appeared nice extensions of Markov's interlacing lemma to more general Extended Chebyshev-Haar systems, see, e.g., \cite{Bojanov2002, BN2002, MN}. In \cite{MN} the authors proved such a result for exponentially weighted polynomials.

\begin{theorem}[Milev-Naidenov]\label{Milev_Naidenov}
Let $\WW_{n}:=\{wp~:~ p\in \Pn\}$, with $w(t)=\exp(-t)$, be the space of exponentially weighted polynomials. Let $f$ and $g$ be two oscillating polynomials\footnote{A function from an ECHS is called an \emph{oscillating polynomial} on $\II$ if it has $n$ simple real zeroes in $\II$.} of $\WW_n$ having zeroes $u_1<\ldots<u_{n}$ and $v_1<\ldots<v_{n}$, respectively and satisfying the interlacing property $u_1\le v_1\le u_2\le\dots\le v_{n-1} \le u_{n}\le v_{n}$.

Then we also have the similar interlacing property for the roots of the derivatives $f', g'$. Moreover, if some of the majorization inequalities $u_i\le v_i$ holds with strict inequality, i.e., $u_i<v_i$, then also the inequalities between the roots of derivatives are strict.
\end{theorem}

However, in the exponentially weighted setting, additional work is needed in order to obtain a complete description of the root interlacing of the derivatives $P_i'$. Indeed, as observed above, the degree of $P_r$ is not necessarily $n$, but it may drop to $n-1$ whenever $a_r=0$. Consequently, in order to apply a Markov-type argument to the pair $(P_i,P_r)$, one needs a corresponding interlacing result in the case when one weighted polynomial is oscillating, while the other has only $n-1$ real and simple zeros (in which case we will call it a ``\emph{nearly oscillating}'' polynomial).

Somewhat surprisingly, such a result is currently known only in the classical algebraic setting and for weighted polynomials with the Hermite weight, where it was established by Milev and Naidenov in \cite[Lemma~4]{MN2006}. Their proof relies on an important idea going back to Bojanov \cite[Lemma~1]{Bojanov1999}: if two strictly oscillating elements of a differentiable ECH-system have ``strictly'' interlacing zeros (meaning that the two sequences of the zeros are interlacing, but not identical), then their derivatives cannot have a common zero. In the Hermite-weighted setting, the corresponding statement appears in \cite[Corollary~3]{MN2006}. In Bojanov's original setting, however, both functions were assumed to be oscillating.

Our next goal is to extend this principle to the present setting, where one function is oscillating while the other is only nearly oscillating. This yields the corresponding extension of the Markov-type theorem for weighted polynomials associated with a general $C^1$ weight, which in particular yields a natural generalization of the result of Milev and Naidenov.

\begin{lemma}\label{l:Bojanov_exp}
Let $w\in C^1(\II)$ be an arbitrary (positive) weight function, and
$
\WW_n:=\{wp:\ p\in \mathcal P_n\}.
$
Suppose that $f,g\in \WW_n$, where $f$ is oscillating with zeros $u_1<\dots<u_n$ and $g$ is nearly oscillating having $n-1$ real and simple zeros $v_1<\dots<v_{n-1}$, which interlace with the zeros of $f$, i.e., $u_1\le v_1\le u_2\le\dots \le u_{n-1}\le v_{n-1}\le u_{n}$.
Then $f'$ and $g'$ have no common zeros.
\end{lemma}

\begin{proof}
Let \(f,g\in \mathcal W_n\) such that the zeros of \(f\) and \(g\) interlace, and assume that there exists at least one index $k$ with $u_j  \notin\{v_1,\ldots, v_{n-1}\}$. To prove that the interlacing of the derivative zeros is strict, we can consider the Wronskian $W(f,g):=fg'-f'g$. Namely, if we can prove that $W(f,g)(t)\ne 0$, for any $t$, then it follows immediately that there is no point $s$ at which both $f'(s)=0$ and $g'(s)=0$, i.e., $f'$ and $g'$ have no common zeros.
Write
\[
f(t)=w(t)P(t), \qquad g(t)=w(t)Q(t),
\]
where
\[
P(t)=\prod_{j=1}^n (t-u_j), \qquad
Q(t)=\prod_{j=1}^{n-1}(t-v_j).
\]
Then, the common factor $ww'PQ$ cancels, and hence
\[
W(t)=w^2(t)\bigl(P(t)Q'(t)-P'(t)Q(t)\bigr),
\]
i.e., it is enough to show that $\frac{W(t)}{w^2(t)}=(P(t)Q'(t)-P'(t)Q(t))\neq 0$.

From here, after removing the weights, the proof is standard, but we briefly include it to make the exposition self-contained. The partial fraction decomposition yields
\begin{equation}\label{part_frac}
\frac{Q(t)}{P(t)}=\sum_{k=1}^n \frac{A_k}{t-u_k}, \qquad \text{where} \quad A_k=\Res\left[\frac{Q}{P};u_k\right]=\frac{Q(u_k)}{P'(u_k)}.
\end{equation}
Moreover, there exists at least one index $j$, such that $u_j\notin \{v_1,\ldots,v_{n-1}\}$, and for any such index we must have $A_j\ne 0$.

By construction, $(-1)^{n-k} P'(u_k)>0$ and $(-1)^{n-k} Q(u_k)\ge 0$  for all $k=1,\ldots,n$, and there exists an index $j$ with $(-1)^{n-j} Q(u_j)> 0$.
Therefore, $A_k=\frac{Q(u_k)}{P'(u_k)}\ge 0$ for all $k=1,\ldots,n$,
and as there is some $j$ with $A_j>0$, too, we are led to
\[
\left(\frac{Q(t)}{P(t)}\right)'= -\sum_{k=1}^n \frac{A_k}{(t-u_k)^2}<0.
\]
Thus we find
\[
\frac{W(t)}{w^2(t)}=P(t)Q'(t)-P'(t)Q(t)=P(t)^2\left(\frac{Q(t)}{P(t)}\right)'<0
\]
for all $t$, which completes the proof.
\end{proof}

We are now in a position to prove the Markov interlacing property for two exponentially weighted polynomials, one having $n$ zeros and the other $n-1$. The argument follows the classical idea used in \cite{Rivlin1974} to extend the original Markov theorem.

\begin{theorem}\label{Markov_exp_Pr}
Let $w\colon \II\to \RR$ be an arbitrary weight such that the spaces $\WW_{n}$, as well as $\W_{n-1}$ of weighted polynomials of degree at most $n$ resp. $n-1$ have Property (P). Let $f\in \WW_n$ be an oscillating polynomial and let $g\in\WW_{n-1}$ have $n-1$ real roots in $[0,\infty)$.
Suppose that the zeroes $u_i$ of $f$ and $v_j$ of $g$ interlace: $u_1\le v_1\le u_2\le\dots \le u_{n-1}\le v_{n-1}\le u_{n}$. Then, we also have a similar, but always strict interlacing property for the roots of $f'$ and $g'$.
\end{theorem}

\begin{proof}
Let $f(t)=w(t)\prod\limits_{j=1}^{n}(t-u_j)$ and $g(t)=w(t)\prod\limits_{j=1}^{n-1}(t-v_j)$.
We introduce $\tilde{g}(t):=g(t)(t-v_n)$, where $v_n$ is taken to satisfy $u_n<v_n$ (and we will consider eventually $v_n\to \infty$). Then, taking into account cancellation of the matching main terms, $\tilde{g}-f\in\WW_{n-1}$, a degree one smaller exponential polynomial. Thus the degree $n-1$ Lagrange interpolation projection, $\LL_{n-1}(w,\uu,\cdot)$-- the fundamental functions now denoted $h_k$ similarly as usually for degree $n$--can be used to represent $\tilde{g}-f\in\WW_{n-1}$. This yields
\begin{equation}\label{g-f}
(\tilde{g}-f)(t)=\sum\limits_{k=1}^{n} (\tilde{g}-f)(u_k) h_k(\uu,t).
\end{equation}
Observe that the interpolatory basis functions can be written as
$$
h_k(\uu,t)=\dfrac{f(t)/(t-u_k)}{w(u_k) \prod_{j=1,j\not=k}^{n}(u_k-u_j)} = \frac{f(t)}{f'(u_k)(t-u_k)}.
$$
Thus we get that
\begin{equation*}\label{dist_gf}
(\tilde{g}-f)(t)=\sum\limits_{k=1}^{n} \frac{(\tilde{g}-f)(u_k)}{f'(u_k)}\cdot\frac{f(t)}{t-u_k},
\end{equation*}
in other words
\begin{equation}\label{dist_gf2}
g(t)(t-v_n)-f(t)=\sum\limits_{k=1}^{n} \frac{g(u_k)(u_k-v_n)-f(u_k)}{f'(u_k)}\cdot\frac{f(t)}{t-u_k} = \sum\limits_{k=1}^{n} \frac{g(u_k)(u_k-v_n)}{f'(u_k)}\cdot\frac{f(t)}{t-u_k}
.
\end{equation}
Note that the formula is certainly valid for $t\ne u_k$, but when $t=u_k$ then it is of the form $0/0$, and so $f(t)/(t-u_k)|_{t=u_k}$ may be interpreted as the limit $f'(u_k)$.

Now, assume that $f'(z)=0$. Then $f(z)\not=0$ (as $f$ has $n$ different, hence simple zeros). Taking the derivative of both sides of \eqref{dist_gf2} and evaluating them at $t=z$, we get that
\begin{equation}\label{g-f_der}
g'(z)(z-v_n)+g(z)=-\sum_{k=1}^{n}\frac{g(u_k)(u_k-v_n)}{f'(u_k)}\cdot\frac{f(z)}{(z-u_k)^2}.
\end{equation}
Note that $(-1)^{n-k} f'(u_k)>0$ and $u_k-v_n<0$, for all $k=1,\ldots,n$. Further, $(-1)^{n-k} g(u_k)\ge 0$, and at least for one index $k$ we necessarily have strict inequality here, given that $g$ has only $n-1$ zeros. As a result, \eqref{g-f_der} guarantees that
\begin{equation}\label{positive}
\frac{g'(z)(z-v_n)+g(z)}{f(z)} >0
\end{equation}
for every zero $z$ of $f'$. Now for every $i=1,\ldots,n-1$ there exists one zero of $f'$ in $(u_i,u_{i+1})$. Let $\tau_i$ denote these roots. Therefore, the sign of $f(\tau_i)$ can be calculated, too and we get that
\begin{equation}\label{sign_f_tau}
\sign(f(\tau_i))=(-1)^{n-i}.
\end{equation}
From here the positivity of the function \eqref{positive} at the zero $z=\tau_i$ of $f'$ implies, that $\sign(g'(\tau_i)(\tau_i-v_n)+g(\tau_i))=(-1)^{n-i}$. Moreover, $g(\tau_i)$ is bounded and if $v_n\to \infty$, then $\tau_i-v_n\to -\infty$, and $g'(\tau_i) \ne 0$ by Lemma \ref{l:Bojanov_exp}, hence we conclude $\sign(g'(\tau_i)(\tau_i-v_n)+g(\tau_i)) =\sign(g'(\tau_i)(\tau_i-v_n))=(-1)^{n-i}$ for all sufficiently large $v_n$, that is,
\begin{equation}\label{sign_g_der_tau}
\sign(g'(\tau_i))=(-1)^{n-i+1}.
\end{equation}
Thus, we obtain at least $n-2$ zeros $\sigma_i\in (\tau_i,\tau_{i+1})$ of $g'$.

From this point on, the proof naturally splits into four cases according to the possible values of $\de_0$ and $\de_n$ from Property (P).

\medskip

\textbf{Case 1:} $\de_0=\de_n=0$.

In this case, $f'$ has no further zeros besides the already identified $\tau_1,\dots,\tau_{n-1}$. Moreover, since $g\in \W_{n-1}$ and Property (P) holds for $\W_{n-1}$ with $\de_0=\de_{n-1}=0$, we also have that $g'\in \W'_{n-1}$ cannot have more than $n-2$ zeros, and therefore the already found zeros
\[
\sigma_1,\dots,\sigma_{n-2},
\qquad
\sigma_i\in (\tau_i,\tau_{i+1}),
\]
are all the zeros of $g'$.

\medskip

\textbf{Case 2:} $\de_0=0$ and $\de_n=1$.

Now $f'$ possesses one additional zero $\tau_n\in (u_n,\infty).$ Furthermore, for every zero of $f'$, including $\tau_n$, we still have the sign property \eqref{sign_f_tau}.
Consequently, equation \eqref{positive}, evaluated at the zero $z=\tau_i$ of $f'$, implies that
\[
\sign\bigl(g'(\tau_i)(\tau_i-v_n)+g(\tau_i)\bigr)=(-1)^{n-i}.
\]
Exactly as before, this yields that the equation \eqref{sign_g_der_tau} is also satisfied for
$i=1,\dots,n$. Therefore, there exists an additional zero $\sigma_{n-1}\in (\tau_{n-1},\tau_n)$ of $g'$.

Since now $\delta_{n-1}=1$, and $g\in \W_{n-1}$, therefore also $g'$ belongs to $\W'_{n-1}$ and  $g'$ can have at most $n-1$ zeros altogether. Thus the zeros already found are all the possible zeros $\sigma_i\in (\tau_i,\tau_{i+1})$ ($i=1,\ldots, n-1$) of $g'$. Therefore, we get the interlacing property of the zeros of $f'$ and $g'$.

\medskip

The remaining two cases, namely when $\de_0=1$, $\de_n=0$, and when $\de_0=\de_n=1$, can be treated in a completely analogous manner. In each case the same conclusion follows.
\end{proof}

In the following let us denote the root sequences, arranged in increasing order, of the derivatives $P_i'$ as $W_i$. Recall that $\# W_i=n$ except possibly for $i=r$, where $\# W_r=n-1$ may occur in case $a_r=0$ and $\deg P_r=n-1$. The main goal of our analysis thus far was to obtain the following result.

\begin{corollary}\label{cor:derivativeroots} If $w(t)=\exp(-t)$ is the exponential weight, then the root sequences $W_i$ of the derivatives $P_i'$ of the weighted polynomials $P_i$ defined by $P_i:=L_{n+1}|_{I_i}$ on $I_i$ strictly interlace.

If $\deg P_r=n$, then the interlacing is in the order $W_r\prec W_{r-1}\prec\dots\prec W_1 \prec W_{n+1} \prec W_n \prec \dots \prec W_{r+1}$.

If $\deg P_r=n-1$, then the ordering remains the same except for $W_r$, which has $\# W_r=n-1$ roots, and these zeroes are strictly interlacing with the $n$ zeros of $W_i$ for any $i \ne r$.
\end{corollary}

\section{Maximum place and maximum value on $I_i$}\label{sec:maximum}

We proceed to find the maximum point of the \emph{weighted} Lebesgue function $L_n$ on each of the intervals $I_i$. That means to find the maximum point of $P_i(\xx,t)=e^{-t}p_i(t)$ on $I_i$. For sure we know, that $P_i(\xx, t)>0$ on $I_i$, for every $i=1,\ldots n+1$. Moreover, $P_i(\xx, x_{i-1})=P_i(\xx, x_{i})=1$ ($i=1,\ldots,n$), and $P_{n+1}(x_n)=1$. Looking for the maximum place and value of $P_i$ on $I_i$ is equivalent to looking for the maximum place and value of
\[
F_i(\xx,t):=\log |P_i(\xx,t)| = \log |w(t)p_i(t)|,
\]
on $I_i=I_i(\xx)$.

The paper \cite{K-AMH} claims without proof that the maximum point of $P_i(\xx,\cdot)$ on $I_i$ is unique and lies in the interior of $I_i$. When we are dealing with an Extended Chebyshev-Haar-System (ECHS) containing the constants, then the second part (interiority) is a result of Kilgore and Cheney, see Lemma 1 in \cite{Kilgore-Cheney}. However, generally this does not extend to all weighted polynomial systems $w\Pn$, if ${\bf 1} \not\in \WW_n$, as concrete counterexamples demonstrate in \cite{SzAMH}. To ensure both uniqueness and interiority, we will make use of an additional property of the weight, easily seen to hold for $w(t)=\exp(-t)$. This property is the log-concavity of the weight, that is, the concavity of $\log w(t)$. We will see shortly that this assumption guarantees uniqueness of the maximum point, so that we can denote it by $z_i(\xx)$. Furthermore, we will indeed have $z_i(\xx) \in \intt I_i$ at least for $i=1, \ldots, n$, even if for $i=n+1$ it still may (and does) fail for certain node systems.

The analysis of this section holds for weighted polynomial systems $w\Pn$ which, on the one hand, are coming from a smooth (twice continuously differentiable) and log-concave weight. Note that $w\Pn$ is a rank $n+1$ ECHS.

\subsection{Strong locally uniform concavity.}
\begin{lemma}\label{l:strict_concave_place_zi}
Let $\WW_n=w\Pn$ be a weighted polynomial system with a twice continuously differentiable and log-concave weight $w$.

Let $K$ be any compact set $K \subset \overline{S}\times [0,\infty)$ with the property that it avoids zeros of $P_i$, that is, $P_i(\xx,t) \ne 0$ for $(\xx,t)\in K$. Then $F_i$ is strictly concave on $K$ in the variable $t$, and there exists a constant $c=c(K)>0$ such that $(F_i)''_{tt}(\xx,t)<-c$ uniformly for $(\xx,t)\in K$.
\end{lemma}
\begin{remark} We assumed twice continuous differentiability of $w$ only for convenience. The reader will have no difficulty in seeing that the assertion remains in effect a.e. with the a.e. existing second derivatives of $F_i$, even when no additional smoothness is imposed on $w$.
\end{remark}

\begin{proof}
Representing the polynomial $p_i$ as the product of its root factors we can write\footnote{Here and everywhere below, the products and sums are extending only to $n-1$, if $i=r$ and $a_r=0$. Note that everywhere we rely on the found root locations and numbers of Section \ref{sec:zeros}, where extensive use of $w\Pn$ being an ECHS was needed.}
$$
p_i(t)=a_i\prod\limits_{j=1}^{n} (t-y_j),
$$
where $y_j$ are the roots of $p_i$. Hence
$$
F_i(\xx,t):=\log |w(t) p_i(t)| = \log w(t) + \log|a_i| + \sum_{j=1}^{n} \log|t-y_j|.
$$
In view of log-concavity of $w(t)$, this is a strictly concave function in each interval between the consecutive zeroes of $p_i$ (that is, of $P_i$), and also in the intervals between $-\infty$ and the first zero, respectively, between the last zero and $\infty$. By assumption, $\left(\frac{w'(t)}{w(t)}\right)' \le 0$, hence we are led to
\begin{equation}\label{Fisecond}
(F_i)''_{tt}(\xx,t) = \left(\frac{w'(t)}{w(t)}\right)' -\sum_{j=1}^{n} \frac{1}{(t-y_j)^2} < 0.
\end{equation}

As $(F_i)_{tt}''$ is continuous, this means that on the compact set $K$ its values must be separated from $0$ by a constant.

In particular, for each $i=1,\ldots, n+1$, we have that $\log P_i$ is strictly concave in $I_i$, and, as a result, it has a unique strict maximum point $z_i\in I_i$.
\end{proof}

So at this point we can define the interval maxima $z_i(\xx)$ as a unique maximum point of $F_i$ or, equivalently, of $P_i$, on $I_i$. Given that $P_i(\xx, x_{i-1})=P_i(\xx, x_{i})(=1)$ for all $i=1,\ldots,n$, strict log-concavity immediately implies that $z_i \in \intt I_i$ whenever $1\le i \le n$. Therefore we obtain the following.

\begin{corollary}\label{c:ziininttIi}
If the weight $w$ is twice continuously differentiable and log-concave, then $z_i(\xx) \in \intt I_i$ for every $i=1,\dots,n$.
\end{corollary}

\subsection{The critical points $w_i(\xx)$.}

We have seen that $P_i>0$ on $I_i$. There is a maximal interval $J_i$, containing $I_i$, such that $P_i(t)>0$ (i.e., $F_i(t)>-\infty$) on $\intt J_i$, but $P_i$ vanishes towards the endpoints of $J_i$.

This is clear if $P_i$ has zeroes in both the intervals $I_{i-1}$ and $I_{i+1}$, so definitely when $2\le i \le n$. (Then we denoted these zeros as $y_{i-1}^{(i)}$ and $y_{i+1}^{(i)}$, respectively.) Then $J_i:=(y_{i-1}^{(i)}, y_{i+1}^{(i)})$.

If $i=1$ and there is a zero $y_0^{(1)}$ of $P_1$ with $y_0^{(1)}<x_0=0$, then $P_1>0$ in $(0,y_2^{(1)})$.  Thus we may consider $J_1:=(y_0^{(1)},y_2^{(1)})$ with $y_2^{(1)}\in I_2$ the first positive zero of $P_1$. However, if $r=1$ and $a_1=0$, then (and only then) we may have no root $y_0^{(1)}$, in which case we simply take $J_1=(-\infty,y_2^{(1)})$.

Further, for $i=n+1$ we take $J_{n+1}:=(y_{n}^{(n+1)},\infty)$ with the last zero $y_{n}^{(n+1)}\in I_{n}$ of $P_{n+1}$.

Finally, we have seen, that if $P_n$ has a zero in $I_{n+1}$, then $J_n:=(y_{n-1}^{(n)},y_{n+1}^{(n)})$ is well-defined and in case $P_{n}$ does not have a zero in $I_{n+1}$, then we simply take $J_n:=(y_{n-1}^{(n)},\infty)$ where $y_{n-1}^{(n)}$ is the last zero of $P_{n}$ in $I_{n-1}$.
(Note that then we must have $r=n$ and unless $\deg P_{n}=n-1$, there must exist some additional extra zero $y_0^{(n)} <0$ of $P_{n}$.)

So we have defined the maximal open intervals $J_i=J_i(\xx)$ for any node systems $\xx \in S$ such that by construction $[x_{i-1},x_{i}]=I_i\subset J_i$, $P_i>0$ on $J_i$, and the limits of $P_i(\xx,\cdot)$ towards the endpoints of $J_i$ are zero. In other words, $F_i=\log P_i$ tends to $-\infty$ towards the endpoints of $J_i$, whilst $F_i$ is strictly concave on $J_i$. Therefore $F_i(\xx,\cdot)$  has a unique strict maximum point $w_i=w_i(\xx)$ inside $\intt J_i$. Moreover, $F_i(\xx,\cdot)$ is differentiable, and the maximum point must satisfy $(F_i)'_t(\xx,w_i(\xx))=0$. In view of strict concavity, we can identify $w_i(\xx)$ as the unique critical point of $F_i$ on $J_i$. When $z_i(\xx) \in \intt I_i(\xx)$ -- in particular, for the exponential weight (or for any log-concave weight) certainly for $1 \le i \le n$ -- the interval maxima $z_i(\xx)$ and the critical points $w_i(\xx)$ coincide. However, for the end interval it is well possible that $y_n^{(n+1)} < w_{n+1} < x_n =z_{n+1}$, and $z_{n+1}$ is \emph{not} a critical point. We only know that $w_{n+1}(\xx) \in J_{n+1}(\xx) = (y_n^{(n+1)},\infty)$.

To proceed, we need to establish a uniform localization of $w_i(\xx)$ independently of the choice of $\xx$, at least locally.

\begin{lemma}\label{l:localization-wi} Let $w$ be a twice continuously differentiable log-concave weight function.

If $\ba \in S$ and $1\le i \le n+1$ is fixed, then there exists $\de>0$ and a compact interval $N=[\alpha, \beta] \subset J_i(\ba)$, such that for $\xx \in B(\ba,\de)$ we also have $w_i(\xx) \in N$.

Moreover, $w_i(\xx)$ is continuously differentiable, and its derivative is uniformly bounded in $B(\ba,\de)$.
\end{lemma}

The assertion is essentially trivial, once we observe that $w_i(\xx)$ is the unique solution of the implicit functional equation $(F_i)'_t(\xx,w_i(t))=0$, and establish that the derivative with respect to $t$, that is, $(F_i)''_{tt}$, is continuous and separated from 0. The below detailed calculus is provided only for certain technical advantages in the further calculations.

\begin{proof} Take the interval $J_i:=J_i(\ba)$, and consider the found maximum $\mu:=\max_{J_i(\ba)} F_i(\ba,\cdot) = F_i(\ba,w_i(\ba))$. As $w_i(\ba) \in \intt J_i(\ba)$, there is a small $\de_1>0$ such that for all $\xx \in B(\ba,\de_1)$ we still have $w_i(\ba) \in \intt J_i(\xx)$ (where $J_i(\xx)$ is the interval that corresponds to $i$ and $\xx$). To see this, note that we chose for endpoints of the interval either zeros or $+\infty$, and the zeros of $P_i(\xx;\cdot)$ change continuously in function of the nodes $x_k$ (whereas $\infty$ remains fixed). Also, for given $\ba$ and with a possibly even smaller $\de_2>0$ we have for all $\xx \in B(\ba,\de_2)$ that $\max_{J_i(\xx)} F_i(\xx,\cdot) \ge F_i(\xx,w_i(\ba))>F_i(\ba,w_i(\ba))-1=\mu-1$.

Let $M<\mu-2$ be a small parameter (which is negative but of large absolute value). We take the level set $N:=\{t \in J_i(\ba)~:~ F_i(\ba,t) \ge M\}$. It is a compact interval. Indeed, since $F_i(\ba,\cdot)$ is continuous, the level set is closed; it is an interval by concavity; and it is bounded because $F_i(\ba,t)\to -\infty$ as $t$ approaches the endpoints of $J_i(\ba)$.

Let the found interval be $N:=[\al,\be]$. By construction, we have $F_i(\ba,\al)=F_i(\ba,\be)=M$, and $w_i(\ba) \in N$. As $F_i$ depends continuously on $\xx$, for  small enough $\de_3>0$ and $\xx \in B(\ba,\de_3)$ we have that $$F_i(\xx,\al) \in (M-1,M+1),\quad F_i(\xx,\be) \in (M-1,M+1),\quad \text{and} \quad F_i(\xx, w_i(\ba)) \in (\mu-1,\mu+1).$$ It then follows from concavity that on $N$ even $F_i(\xx,\cdot)\ge \min(F_i(\xx,\al) ,F_i(\xx,\beta)) \ge M-1$, hence $N\subset J_i(\xx)$, and that $w_i(\xx) \in N$, for $F(\xx,w_i(\xx))=\max_{J_i(\xx)} F_i(\xx,\cdot) \ge F_i(\xx,w_i(\ba)) > \mu-1 >M+1$, whereas on $J_i\setminus N$ we have $F_i(\xx,\cdot) \le \max(F_i(\xx,\al), F_i(\xx,\be)) \le M+1$.
So, the strictly concave function $F_i(\xx,\cdot)$ attains its unique maximum inside $\intt N$ for all $\xx \in B(\ba,\de_3)$, this maximum being the unique maximum both on $J_i(\xx)$ and also on $N\subset J_i(\xx)$.

Recall that the unique strict maximum of $F_i(\xx,\cdot)$ on $N$ can be identified as the unique solution of the derivative equation
\begin{equation}\label{Fprimezero}
(F_i)'_t(\xx,t)=0 .
\end{equation}
We already know that it has a unique solution in $\intt N$ for each $\xx \in B(\ba,\de_3)$; the Implicit Function Theorem tells us that this solution is continuously differentiable and its total derivative (gradient) is given by the formula
\begin{equation}\label{IFTappl}
\triangledown_{\xx} w_i(\xx)=\left(\ldots,\frac{\partial w_i(\xx)}{\partial x_k},\ldots\right)_{k=1}^n= -\frac{1}{(F_i)''_{tt}(\xx,w_i(\xx))} \triangledown_{\xx} (F_i)'_t(\xx,w_i(\xx)).
\end{equation}

By Lemma \ref{l:strict_concave_place_zi}, for $(\xx,t) \in B(\ba,\de_3)\times N$ we have $0 \le \frac{-1}{(F_i)''_{tt}(\xx,t)} \le C$, i.e., it is uniformly bounded. We can as well compute the $\xx$-derivative of $F_i$. For any index $k=1,\ldots,n$ we have
$$
\frac{\partial (F_i)'_t(\xx,t)}{\partial x_k} =
\frac{\partial\frac{(P_i)'_t(\xx,t)}{P_i(\xx,t)}}{\partial x_k} = \frac{ \frac{\partial (P_i)'_t(\xx,t)}{\partial x_k} P_i(\xx,t)-(P_i)'_t(\xx,t) \frac{\partial P_i(\xx,t)}{\partial x_k}}{P^2_i(\xx,t)}.
$$
By construction, on $B(\ba,\de_3) \times N$ the function $P_i(\xx,t)$ is bounded from below by $e^{M-1}$, for $F_i(\xx,\al), F_i(\xx,\be) \ge M-1$, and $F_i$ is concave.  Hence we can estimate the denominator.

It remains to see the derivatives $\frac{\partial P_i(\xx,t)}{\partial x_k} = \sum_{j=0}^n \vji \frac{w(t)}{w(x_j)} \frac{\partial \ell_j(\xx,t)}{\partial x_k} -\ve_{k,i} w(t) \frac{w'(x_k)}{w^2(x_k)} \ell_k(\xx,t)$ and
$
\frac{\partial (P_i)'_t(\xx,t)}{\partial x_k} %&
= \sum_{j=0}^n \vji \left( \frac{w'(t)}{w(x_j)} \frac{\partial \ell_j(\xx,t)}{\partial x_k}+\frac{w(t)}{w(x_j)} \frac{\partial \ell'_j(\xx,t)}{\partial x_k} \right)-\ve_{k,i} \frac{ w'(x_k)}{w^2(x_k)} \left( w'(t) \ell_k(\xx,t) + w(t) \ell'_k(\xx,t) \right).
$

As each $\ell_j$ contains only factors with separated $x_i$ coordinates (which must be close to the coordinates $a_i$), without much further calculus it should be clear that there is some uniform upper bound on these derivatives on the whole of $B(\ba,\de_3)\times N$. Thus we can see that the derivative is indeed meaningful, (locally) uniformly bounded, and even continuous.
\end{proof}

\subsection{Lipschitz continuity of the interval maxima $z_i(\xx)$.}
\begin{lemma}[Lipschitz continuity of the maximum location on $I_i$] Let $w$ be a twice continuously differentiable log-concave weight function.

Then the Lebesgue function $L_n(w,\xx,t)$ has a unique maximum point $z_i(\xx)$ in each intervals $I_i$ for $i=1,\ldots,n+1$. This maximum is a strict maximum for $I_i$ and depends on $\xx$ in a locally Lipschitz continuous manner.
\end{lemma}
\begin{proof} That $z_i(\xx)$ exists and is unique, and moreover it is a strict maximum point for $I_i$, follows directly from the strict concavity of $F_i:=\log P_i$. However, we do not always have that $z_i(\xx)=w_i(\xx)$, as $w_i(\xx)\in N \subset J_i$ may be off $I_i$. In our setup, this may happen for $i=n+1$, when possibly $w_{n+1}<z_{n+1}=x_n$. Nevertheless, for $i=n+1$ it is still obvious that $z_{n+1}(\xx)=\max(x_n, w_{n+1}(\xx))$.

So we recall that if $g, f : K \to \RR$ are Lipschitz continuous functions on a domain $K$, then $h:=\max(f,g)$ is Lipschitz continuous, too.
\end{proof}

\section{Derivatives of the local maxima functions.}\label{sec:Derivativesofmi}

The above has a fundamental consequence: all the $M_i(\xx)=F_i(\xx,z_i(\xx))$ are continuously differentiable.

\begin{proposition}\label{prop:Miderivative}
Let $w$ be a twice continuously differentiable log-concave weight function on $[0,\infty)$.

At any $\ba \in S$ the interval maxima functions $M_i$ are continuously differentiable for all $i=1,\ldots,n+1$, and their $j$th partial derivatives are
\begin{equation}\label{Mipartialderivative}
\frac{\partial M_i}{\partial x_j} (\ba)= \frac{\partial F_i }{\partial x_j} (\ba,z_i(\ba)),  \qquad(i=1,\ldots,n,~ j=1,\ldots,n),
\end{equation}
and for $i=n+1$ and arbitrary $j=1,\ldots,n$
\begin{equation}\label{Mnplus1partialderivative}
\frac{\partial M_{n+1}}{\partial x_j} (\ba) =
\begin{cases}
\dfrac{\partial F_{n+1}}{\partial x_j}(\ba,z_{n+1}(\ba)) \qquad \qquad \qquad \textrm{if} \quad z_{n+1}(\ba) > a_n,
\\
\dfrac{\partial F_{n+1}}{\partial x_j}(\ba,a_n)+(F_{n+1})'_t(\ba,a_n)  \quad \textrm{if} \quad z_{n+1}(\ba)=a_n
\end{cases}.
\end{equation}
In other words,
\begin{equation}\label{Migradient}
\nabla_{\xx} M_i(\ba) = \nabla_\xx F_i(\ba,z_i(\ba)), \qquad \qquad  \textrm{if} \quad i=1,\ldots,n ,
\end{equation}
and
\begin{equation}\label{Mnplus1gradient}
\nabla_{\xx} M_{n+1}(\ba) =
\begin{cases}
\nabla_\xx F_{n+1}(\ba,z_{n+1}(\ba)), \qquad \qquad \textrm{if} \quad z_{n+1}(\ba)>a_n,
\\ \nabla_\xx F_{n+1}(\ba, a_n)+(F_{n+1})'_t(\ba, a_n) \ee_n, \quad \textrm{if} \quad z_{n+1}(\ba)=a_n.
\end{cases}
\end{equation}
\end{proposition}

Note that the above formulae describe a continuous derivative, since whenever $z_{n+1}(\xx) \in \intt I_{n+1}$, then we have $(F_{n+1})'_t(\xx,z_{n+1}(\xx))=0$, hence when $\xx \to \ba$ with $z_{n+1}(\xx)\in \intt I_i(\xx)$ but $z_{n+1}(\ba)=a_n$, then also the endpoint value of $(F_{n+1})'_t(\ba,z_{n+1}(\ba))=(F_{n+1})'_t(\ba, a_n)$ is equal to $\lim_{\xx \to \ba} (F_{n+1})'_t(\xx,z_{n+1}(\xx))=0$ by continuity.

\begin{proof} For $i=1,\ldots,n$, and in the first of the two cases occurring when $i=n+1$, the argument is the same. As we have seen, for an interior point $z_i(\ba) \in \intt I_i(\ba)$ the point $z_i(\ba)=w_i(\ba)$ is a solution of the functional equation \eqref{Fprimezero}. Moreover, $z_i(\ba)$ is continuously differentiable with a bounded gradient $\nabla_{\xx} z_i(\ba)=\nabla_{\xx} w_i(\ba)$ at $\ba$, which was already computed in \eqref{IFTappl}. Therefore, a simple chain rule provides
$$
\nabla_{\xx} M_i(\ba) = \nabla_\xx F_i(\ba,z_i(\ba)) + (F_i)'_t(\ba,z_i(\ba)) \nabla_\xx z_i(\ba)  = \nabla_\xx F_i(\ba,z_i(\ba))+\bf{0}.
$$

It remains to settle the second possibility for $i=n+1$, when $z_{n+1}(\ba)=a_n$. Then $z_{n+1}(\ba)$ may not be differentiable, but we can still compute
\begin{align*}
M_{n+1}(\xx) - M_{n+1}(\ba) = & F_{n+1}(\xx,z_{n+1}(\xx)) - F_{n+1}(\ba,z_{n+1}(\ba))
\\= & F_{n+1}(\xx, z_{n+1}(\xx)) - F_{n+1}(\xx, z_{n+1}(\ba))
%% \\ &
+ F_{n+1}(\xx,z_{n+1}(\ba)) - F_{n+1}(\ba,z_{n+1}(\ba))
\\= & (F_{n+1})'_t(\xx,\zeta) \cdot (z_{n+1}(\xx) - z_{n+1}(\ba))
\\ &
\qquad \qquad + \nabla_\xx F_{n+1} (\ba, z_{n+1}(\ba)) \cdot (\xx-\ba)+o(\|\xx-\ba\|),
\end{align*}
where $\zeta \in [z_{n+1}(\ba),z_{n+1}(\xx)]$. Note that these endpoints are both in the Lipschitz domain of $z_{n+1}$. This can be seen by properly adjusting the parameter $M$ in the above argument\footnote{Indeed, there we could take any $M<\mu-2$ with $\mu:=F_i(\ba,w_i(\ba))$, and here we only need to further restrict $M$ requiring additionally $M<F_i(\ba,z_i(\ba))-2$.} for the proof of Lemma \ref{l:localization-wi}, and then noting that for setting $M$ small enough, we can guarantee $z_{n+1}(\xx) \in N$, too, for all $\xx \in B(\ba,\de) \subset B(\ba,\de_3)$. In particular, $\zeta \in [z_{n+1}(\ba),z_{n+1}(\xx)] \subset N$ and $(\xx,\zeta), (\ba,z_{n+1}(\ba)) \in B(\ba,\de)\times N$. So by Lipschitz continuity of $z_{n+1}$ and with an appropriate Lipschitz constant $K$ we have
$$
\left|[(F_{n+1})'_t(\xx,\zeta)-(F_{n+1})'_t(\ba,z_{n+1}(\ba)) ]\cdot (z_{n+1}(\xx) - z_{n+1}(\ba)) \right| \le o(1) K \|\xx-\ba\|,
$$
given that $(F_{n+1})'_t(\xx,\zeta) \to (F_{n+1})'_t(\ba,z_{n+1}(\ba))$ in view of continuity of $(F_{n+1})'_t$ on $B(\ba,\de)\times N$.

Winding up the above it follows
\begin{align}\label{Mnplus1minusMn} \notag
M_{n+1}(\xx) - M_{n+1}(\ba) = & (F_{n+1})'_t(\ba,z_{n+1}(\ba)) \cdot (z_{n+1}(\xx) - z_{n+1}(\ba))
\\ & \qquad \qquad + \nabla_\xx F_{n+1} (\ba, z_{n+1}(\ba)) \cdot (\xx-\ba)+o(\|\xx-\ba\|).
\end{align}
Here we encounter a difference of $(z_{n+1}(\xx) - z_{n+1}(\ba))$, whereas $\xx \to \ba$. We know that $z_{n+1}(\ba)=a_n$. If the same holds for all $\xx$ close enough to $\ba$, then we have here $x_n-a_n$, whose gradient with respect to the $x_j$ is $\nabla_{\xx} (x_n-a_n)=\ee_n$. If on the other hand there exist arbitrarily close $\xx \sim \ba$, with $z_{n+1}(\xx)>x_n$, then at these points we must have $(F_{n+1})'_t(\xx,z_{n+1}(\xx))=0$, for $z_{n+1}(\xx)$ is an interior maximum point in $I_{n+1}(\xx)$. So, in this case also $(F_{n+1})'_t(\ba,z_{n+1}(\ba))=0$ by continuity, and writing $(F_{n+1})'_t(\ba,z_{n+1}(\ba)) \cdot (z_{n+1}(\xx) - z_{n+1}(\ba)) = (F_{n+1})'_t(\ba,z_{n+1}(\ba)) \ee_n \cdot (\xx-\ba)$ is a valid statement.

In both cases we find that under the condition of $z_{n+1}(\ba)=a_n$, the right hand side of \eqref{Mnplus1minusMn} becomes
$$
\left[ (F_{n+1})'_t(\ba,z_{n+1}(\ba)) \ee_n + \nabla_\xx F_{n+1} (\ba, z_{n+1}(\ba)) \right] \cdot (\xx-\ba)+o(\|\xx-\ba\|).
$$
The assertion follows.
\end{proof}

\begin{proposition}\label{prop:Mnplus1derivative}
Let $w$ be a twice continuously differentiable log-concave weight function. If for an $\ba \in S$ $z_{n+1}(\ba)=a_n$ holds, then we also have
$$
\dfrac{\partial M_{n+1}}{\partial x_j} (\ba) =0 \qquad (j=1,\ldots,n).
$$
That is, we also have
\begin{equation}\label{Mnplus1gradientzero}
\nabla_{\xx} M_{n+1}(\ba) = {\bf 0}, \qquad \textrm{if} \quad z_{n+1}(\ba)=a_n.
\end{equation}
\end{proposition}

\begin{proof} The previous proposition established the continuous differentiability of $M_i$ with respect to $\xx$, a nontrivial statement, which required a proof.

However, once we know that, we can argue by continuity in finding $\frac{\partial  M_{n+1}}{\partial x_j} (\ba)$. Recall that $\intt X^c= \{ \ba \in S~:~ \exists ~\de>0 \quad \text{such that} ~ \forall ~ \xx \in B(\ba,\de) ~z_{n+1}(\xx)=x_n\}$.
Now on $\intt X^c$ we have $M_{n+1}(\xx)=F_{n+1}(\xx,x_n)=0$, identically, hence $M_{n+1}$ is constant and its gradient is ${\bf 0}$. By continuity of $\nabla_\xx M_{n+1}$, this nullity extends to the closure $\overline{\intt X^c}$.

However, we have seen in Proposition \ref{prop:fat} that $X^c$ is a fat and closed set, so that $\overline{\intt X^c}=X^c$. The assertion follows.
\end{proof}

Now, we can compute the derivatives of the local maxima functions of $P_i(\xx,\cdot)$ on the interval $I_i(\xx)$. Recall $m_i(\xx)=\exp(M_i(\xx))$, and hence $m_i(\xx)=\exp(F_i(\xx,z_i(\xx)))=P_i(\xx,z_i(\xx))=\max_{I_i(\xx)}P_i(\xx,.)$. Then, from the logarithmic derivative formulae in Proposition~\ref{prop:Miderivative} and Proposition~ \ref{prop:Mnplus1derivative} we also have the following.

\begin{proposition}\label{prop:miderivative}
Let $w$ be a twice continuously differentiable log-concave weight function.

At any $\ba \in S$ the interval maxima functions $m_i$ are continuously differentiable for all $i=1,\ldots,n+1$, and their $j$th partial derivatives are
\begin{equation}\label{mipartialderivative}
\frac{\partial m_i}{\partial x_j} (\ba)= \frac{\partial P_i}{\partial x_j} (\ba,z_i(\ba)),  \qquad(i=1,\ldots,n,~ j=1,\ldots,n),
\end{equation}
and for $i=n+1$ and arbitrary $j=1,\ldots,n$
\begin{equation}\label{mnplus1partialderivative}
\frac{\partial m_{n+1}}{\partial x_j} (\ba) =
\begin{cases}
\frac{\partial P_{n+1}}{\partial x_j}(\ba,z_{n+1}(\ba)), \qquad \qquad \qquad & \textrm{if} \quad z_{n+1}(\ba) > a_n,
\\[2mm]
\frac{\partial P_{n+1}}{\partial x_j} (\ba,a_n) +(P_{n+1})'_t(\ba,a_n)=0, \quad
& \textrm{if} \quad z_{n+1}(\ba)=a_n.
\end{cases}
\end{equation}
In other words,
\begin{equation}\label{migradient}
\nabla_{\xx} m_i(\ba) = \nabla_\xx P_i(\ba,z_i(\ba)), \qquad (i=1,\ldots,n) ,
\end{equation}
and
\begin{equation}\label{mnplus1gradient}
\nabla_{\xx} m_{n+1}(\ba) =
\begin{cases}
\nabla_\xx P_{n+1}(\ba,a_{n}), \qquad &\textrm{if} \quad z_{n+1}(\ba)>a_n,
\\[2mm] \nabla_\xx P_{n+1}(\ba,a_n)+(P_{n+1})'_t(\ba,a_n) \ee_n={\bf 0}, \quad &\textrm{if} \quad z_{n+1}(\ba)=a_n.
\end{cases}
\end{equation}
\end{proposition}

\section{Consequences for the signs of partial derivatives}\label{sec:derivatives}

\begin{corollary}[Magic formula]\label{cor:magic}
Let $w$ be a twice continuously differentiable log-concave weight function. Then for every $1\le j\le n$ and $1\le i\le n$ it holds
\begin{equation}\label{der_identity}
\frac{\partial m_i}{\partial x_j}(\xx)=-h_j(\xx,z_i(\xx))(P_i)'_t(\xx,x_j),
\end{equation}
and for the derivative of $m_{n+1}$, we have for every $1\le j\le n$ that
\begin{equation}\label{der_final_n}
\frac{\partial m_{n+1}}{\partial x_j}(\xx)=\begin{cases}
-h_j(\xx,z_{n+1}(\xx))(P_{n+1})'_t(\xx,x_j) & \text{ if } z_{n+1}(\xx)\in \intt I_{n+1},
\\ 0  & \text{ if } z_{n+1}(\xx)=x_n.
\end{cases}
\end{equation}
\end{corollary}
\begin{proof}
We apply Lemma \ref{l:magic} for arbitrary $i=1,\ldots,n+1$ and $j=1,\ldots,n$. Through this, \eqref{mipartialderivative} furnishes the assertion \eqref{der_identity} for $i=1,\ldots n$, and the first part of \eqref{mnplus1partialderivative} yields the first part of \eqref{der_final_n} if $i=n+1$ and $z_{n+1}(\xx)>x_n$. Finally, for the second part of \eqref{der_final_n} -- i.e., for the case when $z_{n+1}(\xx)=x_n$ -- we can directly refer to the last equality of \eqref{mnplus1gradient}.
\end{proof}

Kilgore and Cheney worked out the $j=i$ and $i-1$ cases of the above in order to show that $m_i$ and $m_{i+1}$ are monotone in the opposite direction in the variable $x_i$. However, Lemma 3 in \cite{Kilgore-Cheney} uses that $\II=[a,b]$, $\WW_n$ contains the constant functions and the endpoints $a$ and $b$ are among the interpolation nodes. Compare also Lemma 7 of \cite{Kilgore-Cheney}. So we give here the version reshaped to our needs. Recall the definition of the subsets $X:=\{z_{n+1}(\xx) > x_n\}$ and $X^c$ its complement in $S$.

\begin{corollary}\label{c:countermonotonicity}
Let $w$ be a twice continuously differentiable log-concave weight function.

If $\xx\in X$, then
\begin{equation}\label{signchange}
\frac{\partial m_i}{\partial x_i}(\xx)>0  \qquad \text{ and } \qquad
\frac{\partial m_{i+1}}{\partial x_i}(\xx)<0, \qquad (i=1,\dots,n).
\end{equation}

If $\xx \in X^c$, then \eqref{signchange} still holds for every $i<n$, but for $i=n$ we have
\begin{equation}\label{signchangeXc}
\frac{\partial m_n}{\partial x_n}(\xx)>0  \qquad \text{ and } \qquad
\frac{\partial m_{n+1}}{\partial x_n}(\xx)=0.
\end{equation}
\end{corollary}

\begin{proof}
First let $\xx \in X$. By Lemma \ref{l:strict_concave_place_zi}, we know that $z_i(\xx)\in \intt I_i$ for every $\xx\in S$ and $i=1,\dots,n$, and for $\xx\in X$ we even have $z_{n+1}\in \intt I_{n+1}$, too. So, applying Lemma \ref{l:magic}, we get that
\begin{equation}\label{magicapplied}
\begin{aligned}
\frac{\partial m_i}{\partial x_i}(\xx)&=-h_i(\xx,z_i(\xx))(P_i)'_t(\xx,x_i),
\\
\frac{\partial m_{i+1}}{\partial x_i}(\xx)&=-h_i(\xx,z_{i+1}(\xx))(P_{i+1})'_t(\xx,x_i).
\end{aligned}
\end{equation}
We also know, that the function $h_i(\xx,.)$ is positive on $\intt (I_i\cup I_{i+1})$ --see its definition in \eqref{eq:hkt} -- and thus $\sign h_i(\xx,z_i(\xx))=\sign h_i(\xx,z_{i+1}(\xx))=+1$.

Moreover, the following properties hold for the function $P_i$ and its derivative $(P_i)'_t$ on $I_i$ for all $i=1,\ldots,n$. $P_i$ is strictly log-concave on $I_i$;  $P_i(\xx,x_{i-1})=P_i(\xx,x_i)=1$; $ P_i(\xx,t)>1$ for every $t \in (x_{i-1},x_i)$; and $\restr{(P_i)'_t}{I_i}(t)=0 \Leftrightarrow t=z_i(\xx)$.
Consequently, we get that $(P_i)'_t(\xx,x_i)<0$ ($i=1,\ldots,n$) and $(P_{i+1})'_t(\xx,x_i)>0$ ($i=1,\ldots,n-1$). Furthermore, in view of $\xx \in X$ and $z_{n+1}(\xx)>x_n$, strict log-concavity furnishes the last inequality even for $i=n$. So, finally, these together with $\sign h_i(\xx,z_i(\xx))=\sign h_i(\xx,z_{i+1}(\xx))=+1$ imply \eqref{signchange} in view of \eqref{magicapplied}.

Let now $\xx \in X^c$. The calculation showing that $\dfrac{\partial m_i}{\partial x_i}(\xx)>0$
remains in effect for every $i=1,\ldots,n$, hence in particular for $i=n$, too. However, for $i=n+1$ the identity $\dfrac{\partial m_{n+1}}{\partial x_n}(\xx)=0$ replaces \eqref{magicapplied} in view of the second part of \eqref{der_final_n}, thus leading to the second half of \eqref{signchangeXc} directly.
\end{proof}

Observe that in particular we obtain the following corollary, which will gain importance in Section \ref{sec:connected}.

\begin{corollary}\label{c:n-n-monotonicity} Let $w$ be a twice continuously differentiable log-concave weight function.

Then $m_{n+1}(\xx)$ is monotonically non-increasing with $0=x_0<x_1<\dots<x_{n-1}$ fixed and $x_n$ changing in $(x_{n-1},\infty)$.
\end{corollary}

The following statement also obtains from Lemma \ref{l:magic}, but to read it down is a little trickier. This we will need in Section \ref{sec:majorization}.

\begin{lemma}\label{l:partialmixn} Let $w$ be a twice continuously differentiable log-concave weight function.

If $\xx \in X^c$, then we have that $\dfrac{\partial m_i}{\partial x_n}(\xx) >0$ for all $i=1,\ldots,n$.
\end{lemma}
\begin{proof} Let $\xx$ be an element of $X^c$, i.e., $z_{n+1}(\xx)=x_n$, $m_{n+1}(\xx)=1$. Then $P_{n+1}(\xx,x_n)=1$ and $P_{n+1}(\xx,t)\le 1$ for $t>x_n$. It follows that $(P_{n+1})'_t(\xx, x_n)\le 0$. Consider any other $P_i$ with $1 \le i\le n$. Then for $t>x_{n}$, $|P_i(\xx,t)| \le P_{n+1}(\xx,t)$ by construction. Consider derivatives close to $x_n$, where we have also $|P_i|\approx 1$, hence $|P_i|$ differentiable. Note that $|\cdot|'=\sign$ and hence $|P_i|'_t(\xx, x_n)=\sign(P_i(\xx, x_n)) (P_i)'_t(\xx, x_n)$. We therefore obtain
\begin{align*}
|P_i|'_t(\xx, x_n) & -(P_{n+1})'_t(\xx, x_n)  = \sum_{j=0}^n \left(\sign(P_i(\xx,x_n)) \vji - \ve_{j,n+1}\right)h_j'(x_n)
\\ & = \sum_{j=0}^n \left((-1)^{n-i} (-1)^{i+j+1+\chi_{i\le j}} - (-1)^{j+n+1+1+\chi_{n+1\le j}}\right)h_j'(x_n)
\\ & = \sum_{j=0}^n \left((-1)^{\chi_{i\le j}} +1\right)(-1)^{n+j+1} h_j'(x_n)=2 \sum_{j=0}^{i-1} (-1)^{n+j+1} h_j'(x_n).
\end{align*}
Now, $h_j$ has a maximal number of $n$ zeroes at the nodes $x_i$ with $i\ne j$, whence these nodes are simple, and cannot be zeroes of the derivatives $h_j'$. Further, $h_j$ changes sign from positive to negative at $x_{j+1}$, and so on, from $(-1)^{n+j+1}$ to $(-1)^{n+j}$ at $x_n$, so that $\sign h_j'(x_n)= (-1)^{n+j}$. It follows that all the terms in the above sum are negative, and therefore
\begin{equation}\label{sharpderivaltineq}
|P_i|'_t(\xx, x_n)  -(P_{n+1})'_t(\xx, x_n) < 0.
\end{equation}
Thus, we obtained a slightly stronger statement than the mere observation $|P_i|'_t(\xx, x_n)  \le (P_{n+1})'_t(\xx, x_n)$ following from $|P_i|(\xx, x_n) = P_{n+1}(\xx, x_n)=1$ and $|P_i(\xx,t)|\le P_{n+1}(\xx,t)$ for $t>x_n$. Moreover, in this calculation we did not use\footnote{It is true that for $\xx \in X$ we have $(P_{n+1})'_t(\xx, x_n)>0$, and \eqref{sharpderivaltineq} is not suitable for determining the sign of $|P_i|'_t(\xx, x_n)$, but nevertheless it is still a valid formula.} $\xx \in X^c$. Now if $\xx \in X^c$, then $|P_i|'_t(\xx, x_n)< (P_{n+1})'_t(\xx, x_n) \le 0$ (for $P_{n+1}(\xx, x_n)=1$ and $P_{n+1}\le 1$ at interior points of $I_{n+1}$), so that we find $|P_i|'_t(\xx, x_n)<0$.

Next we appeal to Lemma \ref{l:magic}, according to which
$$
\frac{\partial m_i}{\partial x_n} = -h_n(\xx,z_i(\xx))(P_i)'_t(\xx,x_n) \qquad (i=1,\ldots,n).
$$
Note that this formula fails for $i=n+1$ as we have set $\xx \in X^c$, where the partial derivative is just 0. For $i \le n$, however, we know that $h_n>0$ on $\intt I_n$, and then it changes sign at nodes and has constant sign on all further $I_i$, arriving to $(-1)^{n-i}$ on $I_i$ containing $z_i(\xx)$, i.e.,
\begin{equation}\label{hn_sign}
\sign h_n(\xx,z_i(\xx))=(-1)^{n-i}.
\end{equation}

This settles the sign of the first factor. As for the second factor, $$0>|P_i|'_t(\xx, x_n)=\sign (P_i(\xx, x_n)) (P_i)'_t(\xx, x_n) =(-1)^{n-i} (P_i)'_t(\xx, x_n)$$ resulting in $\sign (P_i)'_t(\xx, x_n) = (-1)^{n-i+1}$.

Combining the two findings, we are led to
$$
\sign\left(\frac{\partial m_i}{\partial x_n}\right) =(-1) (-1)^{n-i} (-1)^{n-i+1} =+1,
$$
for every $1 \le i \le n$.
\end{proof}

\section{Linear independence}\label{sec:linearindependence}

To prove nonsingularity of \eqref{Ak}, we shall work with the transpose of the derivative submatrices rather than with the submatrices themselves. While this may appear somewhat unusual at first, we adopt this convention in order to follow the notation used in \cite{K-AMH}.

\begin{definition}\label{Amatrix} For arbitrary $\vv \in S$ we define the column vectors
$$
\ba_i:=\ba_i(\vv):= \left[ a_{j,i} \right]_{j=1}^n :=\nabla_{\xx} m_i(\vv) =\left[\frac{\partial m_i}{\partial x_j}(\vv)\right]_{j=1}^n \qquad (i=1,\ldots,n+1),
$$
and the matrix
$$
A:=A(\vv):=\left[\ba_i(\vv)\right]_{i=1}^{n+1}= \left[\frac{\partial m_i}{\partial x_j}(\vv)\right]_{j=1, i=1}^{n,n+1}.
$$
Further, we denote by $A_k$ the square matrices, obtained by deleting the $k$th column $\ba_k$ from $A$, and for the (Jacobi) determinants of the $A_k$ we write $D_k:=\det A_k$.
\end{definition}

\begin{definition}\label{Cmatrix} For arbitrary\footnote{In fact, we need them only for $\vv \in X^c$ -- but the definition can be put in general as well.} $\vv \in S$ we define the column vectors
$$
\cc_i(\vv):= \left[ c_{j,i} \right]_{j=1}^n := \left[\frac{\partial P_i}{\partial x_j} (\vv,w_i(\vv)) \right]_{j=1}^n ,
$$
and the matrix $C$ obtained from these column vectors:
$$
C:=C(\vv):=\left[c_{j,i} \right]_{j=1, i=1}^{n,n+1}:=\left[\frac{\partial P_i}{\partial x_j} (\vv,w_i(\vv))\right]_{j=1, i=1}^{n,n+1}=\left[\cc_1, \ldots,\cc_n, \cc_{n+1}\right].
$$
Further, we denote by $C_k$ the square matrices, obtained by deleting the $k$th column $\cc_k$ from $C$.
\end{definition}

Note that for $i=1,\ldots,n$ we have $z_i(\xx)=w_i(\xx)$, therefore it is all the same if we consider $z_i(\xx)$ or $w_i(\xx)$. Similarly, $z_{n+1}(\xx)=w_{n+1}(\xx)$ if $z_{n+1}(\xx)>x_n$, i.e., when $\xx \in X$. In these cases Proposition \ref{prop:miderivative} yields
$$
a_{j,i}:=\frac{\partial m_i}{\partial x_j}(\xx)=\frac{\partial P_i}{\partial x_j} (\xx,z_i(\xx))=\frac{\partial P_i}{\partial x_j} (\xx, w_i(\xx))=:c_{j,i} .
%%% -h_j(\xx,z_i(\xx))(P_i)'_t(\xx,x_j).
$$
In particular it follows that for $\xx \in X$ we have $A(\xx)=C(\xx)$, whereas for arbitrary $\xx \in S$ we have $\cc_i(\xx)=\ba_i(\xx)$ for at least $i=1,\dots,n$.

For $i=n+1$, however, by definition $\ba_{n+1}(\xx)= \nabla_\xx m_{n+1}(\xx)$, and it is equal to $\cc_{n+1}$ only if $w_{n+1}(\xx)=z_{n+1}(\xx)$ (i.e., when $\xx \in X$), but is just ${\bf 0}$ when $\xx \in X^c$. In this case the submatrices $A_i(\xx)$ contain a zero column for $i=1,\ldots,n$, and become singular. Nevertheless, as is seen from \eqref{magic-middleform} putting $t=w_{n+1}(\xx)$, we still have
$$
c_{j,n+1}=\frac{\partial P_{n+1}}{\partial x_j} (\xx,w_{n+1}(\xx))=-h_j(\xx,w_{n+1}(\xx))(P_{n+1})'_t(\xx,x_j).
$$
For the following argument, which is an adaptation of the one of \cite{CBoorPinkus} and will use the elegant general formulation of \cite{Kilgore1985}, we therefore can argue by considering $C$ and then handle the changes in case $A \ne C$ (i.e., when the node vector $\xx$ belongs to $X^c$).

We start with
\begin{align*}
c_{j,i}&:=\frac{\partial P_i}{\partial x_j}(\xx, w_i(\xx))=-h_j(w_i)(P_i)_t'(\xx, x_j)
\\&
=\frac{w(w_i) \prod_{l=0}^n(w_i-x_l) ~ (P_i)_t'(\xx, x_j)}{ w(x_j) \prod_{l=0,~l\ne j}^n (x_j-x_l) ~ (x_j-w_i)}.
\end{align*}
Singularity of any of the matrices $C_i$ is equivalent to the singularity of the derived matrix\footnote{This useful reformulation was suggested by Ditrich Braess, and was duly recorded in the papers of Kilgore and Cheney \cite{Kilgore-Cheney}, Kilgore \cite{K-AMH}, and also of deBoor-Pinkus \cite{CBoorPinkus}.} when we divide the $i$th column by $w(w_i) \prod_{l=0}^n(w_i-x_l)$ and multiply the $j$th row by $w(x_j) \prod_{l=0,~l\ne j}^n (x_j-x_l)$, So, equivalently to the aim of determining the singularity properties of the matrices $C_k$, we can restrict to analyse singularity of the sub-matrices $Q_k$, obtained by dropping the $k$th column, of the matrix
\begin{equation}\label{Qmatrix}
Q:=[q_i(x_j)]_{j=1, i=1}^{n,n+1}, \quad \text{where} \quad q_i(t):= \frac{(P_i)_t'(\xx, t)}{t-w_i} \quad(i=1,\ldots,n).
\end{equation}

\begin{lemma}\label{l:Kilgore_assump}
Put $w(t)=\exp(-t)$ and $\W=\Wn=w\Pn$ the arising ECHS.

Then the functions $q_i(t):=\frac{(P_i)_t'(\xx, t)}{t-w_i}$ $(i=1,\dots,n+1)$ satisfy the following properties:
\begin{itemize}
\item[1)] $q_i(w_j)\ne 0$ for every $i, j=1,\dots,n+1$,
\item[2)] $q_i$ has a unique, simple root\footnote{Further, it has one more in $\RR\setminus [w_1,w_{n+1}]$, except when $i=r$ and $a_r=0$ -- but we don't need this later on.} in $[w_{j},w_{j+1}]$ for all $j\ne i-1,i$,
\item[3)] $q_i$ does not have any root in $[w_{i-1},w_{i+1}]$.
\end{itemize}
\end{lemma}

\begin{proof}
The first thing we need to recall from our previous analysis is that the weighted polynomials $P_i$ have interlacing roots. Moreover, except for $P_r$ and in the case when $a_r=0$, all these polynomials have $n$ distinct real roots in $\RR$, and in case $a_r=0$, i.e., if the degree of $P_r$ is less than $n$, then its degree is $n-1$, and $P_r$ has $n-1$ distinct real roots in $\RR$. Furthermore, there are $n$ roots of the derivatives $P_i'$ ($i=1,\dots,n+1$), except for $P_r'$, if $a_r=0$, in which case $P_r'$ has only $n-1$ roots.

So, $P_i$'s ($i\in\{1,\dots,n+1\}\setminus\{r\}$) are oscillating polynomials of the Markov system $\WW_n$. Concerning $P_r$, the same is true, if $a_r\not=0$, and if $a_r=0$, then $P_r$ has $n-1$ roots, (therefore its degree is also $n-1$). Recall that the root sequences of the $P_i$ were strictly interlacing pairwise. Therefore, invoking the assumption $w(t)=\exp(-t)$, the Markov type theorems -- Theorem \ref{Milev_Naidenov} and Theorem \ref{Markov_exp_Pr} -- furnish that roots of the $P_i'$ are interlacing in the same way as the roots of the $P_i$ did.

Here we may use a technical modification to simplify our life. Namely, we can take $W_i$ to denote the root sequences of $P_i'$ for all $i\ne r$, and even for $i=r$ if $a_r \ne 0$; but to make the cardinality of $W_r$ to be $n$ (and interlace with the other $W_i$ the same way as in case $a_r\ne 0$) we may add an extra auxiliary point to $W_r$, chosen smaller than any other element of any of the $W_i$. With this technical modification, we have that each $W_i$ has cardinality $n$, moreover, they interlace in the order $W_r\prec W_{r-1}\prec\dots\prec W_1 \prec W_{n+1} \prec W_n \prec \dots \prec W_{r+1}$.

From here on we can repeat the classical argument going back to Kilgore \cite{Kilgore} and de Boor-Pinkus \cite{CBoorPinkus}. For a general formulation about pairwise interlacing sequences, this is well described, e.g., in \cite{SzAMH}, see Lemma 11 and Lemma 12 there.

It follows that the $q_i$ satisfy the properties 1)-- 3), listed above.
\end{proof}

Once we have the root distribution statements in Lemma \ref{l:Kilgore_assump}, and recall that all the $q_i$ belong to $\WW_{n-1}$ (which is a rank $n$ ECHS), we can apply a generally formulated elegant result\footnote{Note that the proof goes along the classical lines, where the ECH property was clear from degree considerations for the respective polynomials. The elegant proof in the classical case originates in \cite{CBoorPinkus}.} of Kilgore, who proved the following as Proposition 1 in \cite{Kilgore1985}.
%%%that under these assumptions the system $\{q_1,\dots,q_{n+1}\}\setminus \{q_k\}$ is linearly independent. Let us state it here for further reference.

\begin{lemma}[Kilgore]\label{l:Kilgore_orig}
Let $w_1, \dots, w_{n+1}$ be real numbers with $w_1<w_2<\dots< w_{n+1}$ and $q_1,\dots, q_{n+1}$ be functions which lie in a rank $n$ ECHS. Further, assume that $q_1,\dots, q_{n+1}$ satisfy the three conditions 1)-3) listed in Lemma \ref{l:Kilgore_assump}.
Then, for any $k\in \{1,\dots, n+1\}$, the set $\{q_1,\dots, q_{n+1}\}\setminus \{q_k\}$ forms a linearly independent set of functions.
\end{lemma}

From here we need to arrive at linear independence of the column vectors of $C_k$. For that we need again that the functions $q_i$ generate a rank $n$ ECHS.

\begin{lemma}\label{l:keylemma} Let $q_i$ $i=1,\ldots,n+1$ belong to a rank $n$ ECHS and satisfy the three assumptions on their roots as listed in Lemma \ref{l:Kilgore_assump}.

Assume that $\sum_{i=1}^{n+1} \ai q_i(x_j)=0$ for all $j=1,\ldots,n$ whereas $\al_k=0$ for a fixed $1 \le k \le n+1$. Then we must have that $\al \equiv 0$.
\end{lemma}

\begin{proof}
Since $\sum_{i=1, i\ne k}^{n+1} \ai q_i(x_j)=0$ for all $j=1,\ldots,n$, there are at least $n$ different zeros of this function. But the functions $q_i$ belong to a rank $n$ ECHS, hence we must have $\sum_{i=1,i\not= k}^{n+1} \ai q_i\equiv 0$.

From here, linear independence of the system $\{q_i\}_{i=1, i\ne k}^{n+1}$, given by Lemma \ref{l:Kilgore_orig}, guarantees that the only possibility is $\al \equiv 0$.
\end{proof}

Now we arrive at the conclusion of the section.

\begin{corollary}\label{c:nonsingularity} The matrices $C_k(\xx)$ are nonsingular for all $\xx \in S$ and $k=1,\ldots,n+1$.

For $\xx \in X$ the matrices $A_k(\xx)$ are all nonsingular, too.

For $\xx \in X^c$ only $A_{n+1}(\xx)$ is nonsingular, whereas the remaining $A_k(\xx)$ are singular for $k=1,\ldots,n$.
\end{corollary}

\begin{proof} Above we have explained (see suggestion of D. Braess) why through the ``magic formula'' the nonsingularity of the matrices $C_k$ is equivalent to the nonsingularity of the corresponding submatrices $Q_k$ of $Q$ in \eqref{Qmatrix}.

Dropping any column from $C$, or, equivalently from $Q$, the remaining columns belonging to $Q_k$ can have a zero linear combination only with identically zero coefficients $\ai$ ($i=1,\ldots,n+1, i\ne k$), as was seen before in Lemma \ref{l:keylemma}. That is, $Q_k$ are nonsingular matrices.

That means for $\xx\in X$ the nonsingularity of the matrices $A_k$, as for $\xx \in X$ we have $A=C$, hence $A_k=C_k$ and thus $A_k \sim Q_k$.

Also we have seen that at least for $i=1,\ldots,n$ the $z_i$ are in $\intt I_i$, the ``magic formula'' applies and $\ba_i=\cc_i$.  So, even if $\xx\in X^c$, dropping the last column from $A$, we still get $A_{n+1}=C_{n+1}$, and nonsingularity of $Q_{n+1}$ implies $D_{n+1} \ne 0$, too.

Finally, for $\xx \in X^c$ the remaining $A_k$ (i.e., if $k=1,\ldots,n$) have a zero column $\ba_{n+1}={\bf 0}$ in their last column, in view of \eqref{mnplus1gradient} from Proposition \ref{prop:miderivative}. Obviously, this entails the singularity of $A_k$ for $k=1,\ldots,n$ in this case.
\end{proof}

\section{Connectedness of $X$}\label{sec:connected}

Here we will make essential use of Proposition \ref{prop:fat} above to derive connectedness of $X$.

\begin{lemma}\label{l:Xconnected} The set $X\subset S$ is a connected open domain.
\end{lemma}
\begin{proof} We already know that $X$ is open, for it is a level set $z_{n+1}(\xx)>x_n$. Note that $X$ is not just merely relatively open, for  $S$  itself is open, too.

It remains to show that $X$ is connected. Let $\xx, \yy$ be two points from $X$, and denote, for the purposes of this proof, $\wx:=(x_1,\ldots,x_{n-1})\in S^{n-1}$, and similarly, $\wy:=(y_1,\ldots,y_{n-1})\in S^{n-1}$. Let us connect the points $\wx$ and $\wy$ by a straight line segment such that $\ell(s):=(1-s)\wx+s\wy$ ($0\le s\le 1$). Then $f(s):=\phi(\ell(s))-((1-s)x_{n-1}+sy_{n-1})>0$ is a continuous, positive function according to Proposition \ref{prop:fat}.

Let $0<\mu <\min_{0\le s \le 1} f(s)$. This minimum exists, in view of continuity of the function $f$, and by positivity of the function, its minimum is also positive, hence there exists some such $\mu$, too. The choice of $\mu$ guarantees that for any $s \in [0,1]$ we have $0<\mu<f(s)$, i.e.,
\[
(1-s)x_{n-1}+sy_{n-1} < (1-s)x_{n-1}+sy_{n-1} +\mu < \phi(\ell(s)).
\]

Let us draw the straight line segment $L:=[(\wx,x_{n-1}+\mu),(\wy,y_{n-1}+\mu)]$ joining the points $(\wx,x_{n-1}+\mu)$ and $(\wy,y_{n-1}+\mu)$. Parameterising, a point is
\begin{align*}
L(s): & =((1-s)\wx+s\wy, (1-s)x_{n-1}+sy_{n-1}+\mu) = (\ell(s), (1-s)x_{n-1}+sy_{n-1}+\mu).
\end{align*}
According to the above, we have that $L\subset S$ (the last coordinate staying above the last but one for all $s \in [0,1]$), whereas $(1-s)x_{n-1}+sy_{n-1}+\mu < \phi(\ell(s))$ ensures that $L$ runs below $(\ell(s),\phi(\ell(s)))$, the boundary of $X^c$. Therefore, $L \subset X$ in full.

Consider now the three straight line segments $[(\wx,x_n),(\wx,x_{n-1}+\mu)]$, $[(\wy,y_{n-1}+\mu), (\wy,y_n)]$ and $L$.

The first two segments lie below the points $(\wx,x_n^*)$ and $(\wy,y_n^*)$, respectively, because $x_{n-1}+\mu <x_n^*$ and $x_n <x_n^*$, and $y_{n-1}+\mu< y_n^*$ and  $y_n <y_n^*$, respectively. As for $L$, we have seen that it lies below the curve $(\ell(s),\phi(\ell(s)))$ ($0 \le s \le 1$). So, all three segments are below the curve $(\ell(s),\phi(\ell(s)))$ ($0 \le s \le 1$). Thus, the three segments form a broken line connecting $\xx$ and $\yy$, and lying entirely below the graph of $\phi(\ell(s))$, whereas exceeding $(1-s)x_{n-1}+sy_{n-1}$, so that all the points on this broken line are points from $S$ and even from $X$.

So, the arbitrary points $\xx, \yy$ of $X$ are connected by a broken line within $X$, which yields connectedness of $X$.

\end{proof}

\section{Properness of the difference mapping $\Gamma(\xx)$}\label{sec:proper}

Following \cite{CBoorPinkus} we analyse the ``difference mapping'' $\Gamma\colon S\to \mathbb{R}^{n}$ defined by
\begin{equation}\label{gamma}
\Gamma(\xx):=(m_2(\xx)-m_1(\xx), \dots, m_{n+1}(\xx)-m_{n}(\xx)).
\end{equation}
%%%%Following de Boor and Pinkus, who did the same in the classical case, here we aim at proving that $\Gamma$ is a homeomorphism between $S$ and $\RR^n$. To achieve this, we are going to apply an extension of a result of Shi from \cite{Shi}. Besides the connectedness property of $X$, the properness of $\Gamma$ needs to be seen, as well.

Recall that a continuous map $f: X\to Y$ between topological spaces is called \emph{proper}, if $f^{-1}(Q)(\subset X)$ is compact for any compact subset of $Y$.

\begin{lemma}\label{l:gamma_proper_new}
Let $w$ be a continuous, positive weight function on $[0,\infty)$ with $w(t) t^k \to 0$ ($t\to \infty$) for all $k$ (or at least for $k=n$). Then the mapping $\Gamma\colon S\to \mathbb{R}^{n}$, defined by \eqref{gamma}, is proper.
\end{lemma}

\begin{proof} As $m_i(\xx)$ is continuous on $S$ for every $i=1,\dots,n+1$, the continuity of $\Gamma$ is obvious.

Let $Q$ be a compact subset of $\RR^{n}$ and let $W$ denote the preimage of $Q$ under $\Gamma$, i.e., $W:=\Gamma^{-1}(Q)$. By continuity, $W \subset S$ is a (relatively) closed set in $S$.

In the first step we show, that the boundedness of $Q$ implies that $W$ is also bounded\footnote{The whole proof is a mere technicality with the classic paper \cite{CBoorPinkus} essentially having it all. However, in that case there was no boundedness question, given that the base interval $\II$ was taken finite, so that formally we need a proof here. The paper \cite{SzAMH} works this out in a little more complicated case -- here the proof is a little more direct, but the reader can as well accept the assertion without details of proof as it is very similar to \cite{SzAMH}, Theorem 2.}. For definiteness, assume that the coordinates of points from $Q$ all admit the bound $C$, say. If any $\xx \in S$ belongs to $W$, and is thus mapped to a point of $Q$, then for all applicable index $i$ we have $|m_i(\xx)-m_{i-1}(\xx)|\le C$, hence $\max\limits_{i=1,\ldots,n} m_i(\xx) - m_{n+1}(\xx) \le nC$. Given that $m_{n+1}(\xx) \ge 1$, this also means that for all $i=1,\ldots,n$ we must have
\[
\frac{m_i(\xx)}{m_{n+1}(\xx)} \le nC+1.
\]
Let us fix some large number $q>1$, and assume that $\xx \in S$ is such that $x_n>q$. We will
show that for large enough choice of $q$ this point cannot belong to $W$. For sure, there exists a point $s \in [0,1]$ such that $|s-x_i|\ge 1/(2n+1)$ for all index $i$. This point $s \in [0,1]$ belongs to some of the $I_i(\xx)$, but not to $I_{n+1}(\xx)$, for $s \le 1 <q <x_n$. Take the index $i$ with $s \in I_i(\xx)$ and let $j$ be any index with $0\le j \le n$. Then it holds
\begin{align*}
\left|\frac{h_j(\xx,z_{n+1})}{h_j(\xx,s)}\right|&=\frac{w(z_{n+1})}{w(s)} \prod\limits_{l=0,l\not=j}^n \left|\frac{z_{n+1}-x_l}{s-x_l}\right|
\le \frac{w(z_{n+1})}{\min_{[0,1]} w} \frac{(z_{n+1})^n}{1/\left(2n+1\right)^{n}}
= \frac{(2n+1)^n}{\min_{[0,1]} w} w(z_{n+1}) z_{n+1}^n.
\end{align*}
Recall that $z_{n+1}\ge x_n>q$, so if $q$ is chosen large enough, then by $w(t)t^n \to 0$ we are led to
$$
\left|\frac{h_j(\xx,z_{n+1})}{h_j(\xx,s)}\right| <\ve \qquad (j=0,\ldots,n).
$$
Therefore we get
\begin{align*}
m_i(\xx) & \ge P_i(\xx,s)=\sum\limits_{j=0}^n |h_j(\xx,s)|
\ge \frac{1}{\varepsilon}\sum\limits_{j=0}^n |h_j(\xx,z_{n+1})|=\frac{1}{\varepsilon} P_{n+1}(\xx,z_{n+1})=\frac{1}{\varepsilon}m_{n+1}(\xx).
\end{align*}
Clearly, if $\ve <\frac{1}{nC+1}$, then this means $\xx \notin W$.

\medskip
So $W$ is a bounded and relatively closed set; in particular, there is a constant $c$ such that $\|\xx\|\le c$ for all points $\xx \in W$. It remains to see that $W$ is closed in $\RR^n$, too. In other words, we have to show that if a point $\xx \in \partial S$ satisfies $\|\xx\|\le c$, then it cannot be a limit point\footnote{If we have this, then closing $W$ in $\RR^n$ results in $\overline{W} \subset \oS$, but $\partial S \cap \overline{W}=\emptyset$ so that $\overline{W} \subset S$, and $W$ being relatively closed, $W=\overline{W} \cap S = \overline{W}$ proves that $W$ is closed even in $\RR^n$.}  of $W$.

Now  let $\xx \in \DS$ such that $x_{i-1}=x_i$, say. Also let $\de>0$ be small, and assume that $\yy \in S$ is $\de$-close to $\xx$, i.e., $\|\xx-\yy\|<\de$. We take $z_i \in I_i(\yy)$ the maximum point of $P_i(\yy,\cdot)$ on $I_i(\yy)$, and take any other point $s \in [0,1]$ with the property that it is at least $1/(2n+1)$ distance far from any other point of the node system. With the unique index $k$ with $s \in I_k$ we obviously have
\[
m_k(\yy) \ge P_k(\yy,s) = \sum_{0}^n |h_j(\yy,s)|.
\]
The choice of $s$ leads to
\begin{align*}
\frac{|h_j(\yy,s)|}{|h_j(\yy,z_i)|} & = \frac{w(s)}{w(z_i)} \prod_{l=0, l \ne j}^n \left| \frac{s-y_l}{z_i-y_l}\right|
\ge \frac{\min_{[0,1]} w}{\max_{[0,c+\de]} w} \left( \frac{1}{2n+1} \right)^n (c+\de)^{-n+1} \frac{1}{2 \de},
\end{align*}
for the point $z_i$ is in the interval $I_i(\yy)=[y_{i-1},y_i]$ of length at most $2\de$, and at least one of the indices $l=i-1$ or $l=i$ will occur in the product for $l\ne j$, leading to $|z_i-y_l|< 2\de$. So, the above ratio will stay above $C_1/\de$ with some (possibly large) constant $C_1$. Therefore we find
\[
m_k(\yy) \ge \sum_{j=0}^n |h_j(\yy,s)| \ge \frac{C_1}{\de} \sum_{j=0}^n |h_j(\yy,z_i)| = \frac{C_1}{\de} m_i(\yy).
\]
As before, we can make use of the trivial lower estimate $m_i(\yy)\ge 1$ to infer
\begin{align*}
\max\limits_{l=2,\ldots,n+1} |m_l(\yy)-m_{l-1}(\yy)| & \ge \frac{1}{n} |m_k(\yy)-m_i(\yy)|
\ge \frac{1}{n} \frac{|m_k(\yy)-m_i(\yy)|}{m_i(\yy)}
\ge \frac{1}{n} \left( \frac{C_1}{\de} -1 \right).
\end{align*}
This means that for small $\de$ we will have $\| \Gamma(\yy)\| \ge \frac{1}{n} \left( \frac{C_1}{\de} -1 \right) >C$, whereas $Q$ was bounded by $C$, so that $\Gamma(\yy) \notin Q$. That is, $\xx$ cannot be a limit point of $W$. The proof is finished.
\end{proof}

\begin{corollary}\label{c:proper_m}
The mapping $\mm: S\to \RR^{n+1}$ defined by
\[
\mm(\xx)=(m_1(\xx),\dots,m_{n+1}(\xx)),
\]
is a proper mapping from $S$ to $\RR^{n+1}$.
\end{corollary}
\begin{proof}
The proper map $\Gamma$ is a composition of the continuous mappings $\mm: S\to \RR^{n+1}$ and the simple ``difference mapping'' $d: \RR^{n+1}\to \RR^{n}$ given by
\[
d(\mu_1,\dots,\mu_{n+1}) := (\mu_2-\mu_1,\dots,\mu_{n+1}-\mu_n).
\]
That is, $\Gamma(\xx)=(d\circ \mm)(\xx)$.

Let now $Q \subset \RR^{n+1}$ be any compact set, and denote $W:=\mm^{-1}(Q)$. As $d$ is continuous, $Z:=d(Q) \subset \RR^n$ is compact, too.

Given that $\Gamma$ is a proper mapping, this means that $T:=\Gamma^{-1}(Z)$ is also a compact subset in $S$. Note that $W \subset T$ by construction.

By continuity of $\mm$, we also have $W$ closed (in $\RR^n$). Therefore, $W=W \cap T$ is compact, and it is a subset of $S$, as said.

That proves properness of $\mm$.
\end{proof}

\section{The Minimax Equioscillation Theorem}\label{sec:mnimax}

We continue our progress recalling Corollary \ref{c:nonsingularity} for the next steps.
We set out to the minimax problem of finding
\begin{align*}
\mu(w,S)& :=\inf_{\xx \in S} \overline{m}(\xx):=\inf_{\xx \in S} \max_{i=1,\ldots,n+1} m_i(\xx)
= \inf_{\xx \in S} \sup_\II L_n(\xx,\cdot) = \inf_{\xx \in S} \|L_n(w,\xx,\cdot)\|.
\end{align*}

\begin{theorem}\label{th:minimax} Consider the exponential weight $w(t):=\exp(-t)$, and the weighted Lagrange interpolation in the space $\WW_n$.

There exists an extremal minimax node system $\yy \in S$ such that $\|L_n(w,\yy,\cdot)\| \le \|L_n(w,\xx,\cdot)\|$ for any other node system $\xx \in S$.

 Moreover, any such node system belongs to $X$ and equioscillates: $m_1(\yy)=\dots=m_{n+1}(\yy)$.
\end{theorem}
\begin{proof} The first part--existence of a minimax node system in $S$--follows from continuity and properness of the interval maxima mapping $\mm$, see Corollary \ref{c:proper_m}.

First we claim that any minimax node system necessarily belongs to $X$. In other words, we are to prove that a point $\yy \in X^c$ cannot be a minimax point system. Indeed, if $\yy \in X^c$, then $m_{n+1}(\yy)=1$, while all other interval maxima $m_i(\yy)>1$, ($i=1,\ldots,n$), as it was shown using strict log-concavity, see Lemma \ref{l:strict_concave_place_zi}. Therefore, there exists an $\ve>0$ such that all $m_i(\yy)>1+\ve$. It follows from continuity of the $m_i$ that there exists $\de>0$ such that for any $\xx \in S$ with $\|\xx-\yy\|<\de$ we have %% $m_i(\xx)>1+\ve/2$ ($i=1,\ldots,n$) and
$m_{n+1}(\xx)<1+\ve/2$. Furthermore, $D_{n+1}(\yy)\ne 0$ -- that is, the non-singularity of $A_{n+1}(\yy)$ -- entails that in a sufficiently small neighborhood of $\yy$ the mapping $\xx \to (m_1(\xx),\ldots,m_n(\xx))$ acts as a homeomorphism. So, prescribing the values $\mu_i < m_i(\yy)$ ($i=1,\ldots,n$) sufficiently close to the original values of $m_i(\yy)$, there is a node system $\uu \in S$, close to $\yy$ (in particular closer than $\de$) attaining these prescribed values, i.e., with $m_i(\uu)=\mu_i$ ($i=1,\ldots,n$). As $\| \uu-\yy\|<\de$, the value $m_{n+1}(\uu)$ also satisfies  $m_{n+1}(\uu)<m_{n+1}(\yy)+\ve/2 <  \|L_n(w,\yy,\cdot)\|$. It follows that  $\|L_n(w,\uu,\cdot)\| <  \|L_n(w,\yy,\cdot)\|$, so that $\yy$ cannot be a minimax point in $S$.

Thus we find that if $\yy$ is a minimax node system, then $\yy \in X$. However, on $X$ all $D_k\ne 0$. It follows that whenever $\xx\in X$ and $m_k(\xx)< \|L_n(w,\xx,\cdot)\|$ for some index $k$, then we can apply a similar perturbation argument as before, decreasing  $\|L_n(w,\xx,\cdot)\|$ and therefore showing that $\xx$ was not an extremal (minimax) node system. So, for a minimax point $\yy$ it remains the only possibility that $m_i(\yy)$ are all equal, i.e, $\yy$ must be an equioscillating node system.
\end{proof}

We note, that Shi in Lemma 2 of \cite{Shi} proved a statement similar\footnote{In Shi's work \cite{Shi}, $X$ stands for the open simplex built on a finite interval, here denoted by $S$. Our observation was that the result extends to any open domain. In the following we also needed the finding that our $X$ is not only open, but is also connected, i.e., a domain.} to the previous one. He proved that the existence of a minimax node system guarantees the equioscillating property of it as well if $D_k(\xx)\ne 0$ holds for every $\xx\in X$ and $k=1,\cdots,n+1$.

As the proof of our crucial Theorem \ref{t:homeom}, coming next, applies the key steps of Shi's proof, we present the Lemma of Shi together with its proof. Let us start recalling a basic lemma of linear programming.

\begin{lemma}[\cite{Avriel}, Theorem 3.4]\label{th:LP_Avriel}
Let $g, g_1,\cdots, g_m$ be continuously differentiable on an open subset $Y\subset \RR^{n}$. If $\yy\in Y$ is a solution of the problem to
\[
\min g(\xx), \quad \textrm{ subject to } \xx\in Y \quad \textrm{ and } \quad g_i(\xx)\ge 0, \quad i=1, \dots m,
\]
then there exists a vector $\lambda=(\lambda_0, \cdots, \lambda_m)\not={\bf 0}$, $\lambda \ge {\bf 0}$, such that
\begin{equation}\label{lp:eq1}
\lambda_0 \nabla g(\yy)-\sum\limits_{i=1}^m \lambda_i \nabla g_i(\yy)=0,
\end{equation}
and
\begin{equation}\label{lp:eq2}
\lambda_ig_i(\yy)=0, \quad i=1,\dots,m.
\end{equation}

\end{lemma}

\begin{lemma}[\cite{Shi}, Lemma 2]\label{Shi_lemma2}
Let $Y \subset S$ be any open subset of $S$. Assume that $D_k(\xx)\not=0$ for every $\xx\in Y$
and $k=1,\cdots,n+1$.

Then, if $\yy\in Y$ is a (locally) minimax node system, then it must be equioscillating.

Moreover, the signs of the nonzero Jacobian determinants  satisfy $(-1)^{k-i}D_i(\yy)D_k(\yy)>0$.
%%% are of alternating signs at $\yy$: $D_i(\yy) D_{i+1}(\yy)<0$ ($i=1,\ldots,n$).
\end{lemma}

\begin{proof}
We are going to apply Lemma \ref{th:LP_Avriel} for the function $f_0: \RR\times Y\to \RR$ defined by $f_0(\xi,\xx)=\xi$. So, we consider the following linear programming problem:
\begin{equation}
\begin{aligned}\label{linprogr}
\min\limits_{\xi\in \RR,~ \xx\in Y} f_0(\xi,\xx),\quad \text{ subject to } \quad
%%% \\
f_0(\xi,\xx)-m_i(\xx)\ge 0, \quad (i=1, \dots n+1).
\end{aligned}
\end{equation}

Let $\overline{m}(\xx):=\max\limits_{i=1,\ldots,n+1} m_i(\xx)$ and assume that $\yy$ is a minimax node, that is $\overline{m}(\yy)=\min\limits_{\xx\in Y} \overline{m}(\xx)= \min\limits_{\xx\in Y} \max\limits_{i=1,\ldots,n+1} m_i(\xx)$. It is easy to check that $\yy$ is a minimax node system if and only if $(\overline{m}(\yy),\yy)$ is a solution of our linear programming problem \eqref{linprogr}. So, we can apply Lemma \ref{th:LP_Avriel} to get that there exists a vector $(\lambda_0,\lambda_1,\dots,\lambda_{n+1})\not={\bf 0}$ with non-negative coordinates, such that
\begin{equation}\label{cond_lp_1}
\lambda_0 \nabla_{\xi,\xx} f_0(\overline{m}(\yy),\yy)-\sum\limits_{i=1}^{n+1} \lambda_i \nabla_{\xi,\xx} \left[f_0(\overline{m}(\yy),\yy)-m_i(\yy)\right]=0,
\end{equation}
\begin{equation}\label{cond_lp_2}
\lambda_i\left[f_0(\overline{m}(\yy),\yy)-m_i(\yy) \right]=0, \quad i=1,\dots,n+1.
\end{equation}

Taking the derivative in \eqref{cond_lp_1} with respect to $\xi$, respectively, $\xx$ we get
\begin{equation}\label{lambda0}
\lambda_0-\sum_{i=1}^{n+1} \lambda_i=0,
\end{equation}
and the system of $n$ linear equations with $n+1$ unknown parameters
\begin{equation}\label{lineq}
\sum_{i=1}^{n+1} \lambda_i \frac{\partial m_i}{\partial x_j}(\yy)=0, \qquad j=1,\dots,n,
\end{equation}
respectively.

Recalling that $(\lambda_0,\lambda_1,\dots,\lambda_{n+1})\not={\bf 0}$ and has only non-negative coordinates, \eqref{lambda0} implies that there exists an index $k>0$ with $\lambda_k>0$. As we deal with a homogeneous linear system of equations, dividing by its value we can even assume $\lambda_k=1$. Now, rearranging \eqref{lineq} we get the following non-homogeneous system of linear equations
\begin{equation}\label{inh_lineq}
\sum_{i=1, i\ne k}^{n+1} \lambda_i \frac{\partial m_i}{\partial x_j}(\yy)=-\frac{\partial m_k}{\partial x_j}(\yy), \quad j=1,\dots,n.
\end{equation}
The coefficient matrix of \eqref{inh_lineq} is $A_k$, and it is nonsingular as $D_k\ne 0$. Therefore, applying Cramer's rule we get that
\begin{equation}\label{lambda_i}
\lambda_i=\frac{\det B_i^{(k)}}{D_k}
\end{equation}
where $B_i^{(k)}$ is the modified coefficient matrix obtained from $A_k$ by replacing its $i$th column with the right-hand side vector of \eqref{inh_lineq}. Observe that this vector is just $-\ba_k$, i.e., the negative  of the $k$th column of $A$, so that $B_i^{(k)}$ has essentially the same columns as $A_i$. The only differences being that first, the order of columns of $A_i$ are rearranged, more precisely, leaving everything else on its place, the $k$th column of $A_i$ is now put in the $i$th place, and, second, it is equipped with a negative sign. From this by an application of $|k-i|-1$ changes of neighboring columns and turning back the negative sign, it is easy to compute the value $\det B_i^{(k)}=(-1)^{|k-i|}\det A_i = (-1)^{k-i} D_i$. So we find
\begin{equation}\label{lambda_i JiJk}
\lambda_i=\frac{\det B_i^{(k)}}{D_k}=\frac{(-1)^{k-i}D_i}{D_k}.
\end{equation}
By assumptions of the Lemma, $D_i\not=0$, $i=1,\dots,n+1$ also holds, hence we get that not only $\lambda_k>0$, but every coordinates $\lambda_i \ne 0$ ($i=1,\ldots,n+1$) are nonzero, so in view of non-negativity of $\lambda$, $\lambda_i>0$, for every $i=1,\dots,n+1$. Note that this also determines the relative signs of the determinants $D_i$: $\sign D_i/D_k= (-1)^{k-i}$ for each $1\le i\ne k \le n+1$.

Finally, the strict positivity of all the $\lambda_i$ and \eqref{cond_lp_2} guarantees, that $m_i(\yy)=\overline{m}(\yy)$ for every $i=1,\dots,n+1$ and hence $\yy$ is equioscillating.
\end{proof}

\section{The Homeomorphism Theorem}\label{sec:homeomorphism}

Our goal in this section is to prove the following result, which is fundamental for our purposes. In the classical polynomial interpolation case, the respective result, Theorem 2 of \cite{CBoorPinkus}, was considered by de Boor and Pinkus as their main achievement. In particular, it contains a strong form of the uniqueness of the equioscillating node system.

\begin{theorem}\label{t:homeom}
The difference mapping $\Gamma\colon S\to \mathbb{R}^{n}$, defined by \eqref{gamma}, is a homeomorphism.
\end{theorem}

\begin{proof}
According to Lemma \ref{l:gamma_proper_new}, $\Gamma$ is a proper mapping.
As $S$ is a connected open set (an open domain), and the image space is $\RR^n$, knowing the properness a well-known theorem of topology going back to Hadamard says that $\Gamma$ is a global homeomorphism, if and only if it is a local homeomorphism.  In turn, local homeomorphism is equivalent, in view of continuous differentiability, the nonvanishing of the Jacobian of the derivative matrix
\[
D\Gamma=\left(\frac{\partial m_{i+1}}{\partial x_j} (\xx)-\frac{\partial m_i}{\partial x_j}(\xx)\right)_{i=1,j=1}^{n,n} = \left[\cdots, {\ba}_{i+1}- {\ba}_i, \cdots \right]_{i=1}^{n},
\]
with the notation of Definition \ref{Amatrix} shortened here as $\ba_i=\ba_i(\xx)$.

\begin{lemma}\label{l:linalg} Let $\vv_i \in \RR^n$ be any vectors for $i=1,\ldots,n+1$ and let $V_k$ denote the determinant of $[\vv_1,\ldots,\vv_{k-1},\vv_{k+1},\ldots \vv_{n+1}]$. Then we have
$$
V:=\det [ \vv_2-\vv_1,\ldots, \vv_{i+1}-\vv_i,\ldots,\vv_{n+1}-\vv_n] = \sum_{k=1}^{n+1} (-1)^{k+1} V_k.
$$
\end{lemma}
\begin{proof} We include this standard linear algebra argument for completeness. So let us write
$$
V=\det[\cdots, {\vv}_{i+1}- {\vv}_i, \cdots]_{i=1}^{n} = \det[\cdots, {\vv}_{i+1}- {\vv}_1, \cdots]_{i=1}^{n}.
$$
When expanding the determinant according to the additive law in columns, we obtain $2^{n}$ determinants with various choices of either $\vv_{k+1}$ or $-\vv_1$ in the $k$th column, but out of these only those with no two equal columns, that is with at most one selections of $-\vv_1$ count (as the remaining ones vanish having two identical columns). So, we get
\begin{align*}
V & =\det[\vv_2, \cdots, {\vv}_{i+1}, \cdots, \vv_{n+1}]
+\sum_{k=1}^{n} \det[\vv_2 \cdots, \vv_k, -\vv_1, \vv_{k+2}, \cdots, \vv_{n+1}]
\\& = V_1 +\sum_{k=1}^{n} - (-1)^{k-1} V_{k+1}=\sum_{k=1}^{n+1} (-1)^{k+1} V_k .
\end{align*}
\end{proof}

\emph{Continuation of the proof of Theorem \ref{t:homeom}.} So, in view of this lemma and with the notations of Definition \ref{Amatrix} we have
\begin{equation}\label{DGsignedsum}
\det D\Gamma(\xx) = \sum_{k=1}^{n+1} (-1)^{k+1} D_k(\xx).
\end{equation}

\medskip
Next we are to show $\det D\Gamma \ne 0$ separately both for $X$ and for $X^c$.

\medskip
As for $X^c$, the situation is relatively simple. According to Corollary \ref{c:nonsingularity}, all $A_k$ are singular (with the last column being identically zero) except for $A_{n+1}$, which is not. So for the respective Jacobian determinants, $D_k(\xx)=0$, $k=1,\ldots,n$, while $D_{n+1}(\xx)\ne 0$, whenever $\xx\in X^c$.
Therefore we have that $\det D\Gamma(\xx)=(-1)^{n+2}D_{n+1}(\xx)\ne 0$ on $X^c$.

\medskip
Now we are to prove $\det D\Gamma(\xx)\not=0$ on $X=\{\xx\in S: z_{n+1}(\xx)>x_n\}$. We follow the ideas of Shi (Lemma 3 in \cite{Shi}) with slight modifications. Note that in \cite{Shi} $X$ stands for the whole set $S$, and a number of other conditions, including properness \emph{on $X$}, is required. So, in particular, in \cite{Shi} the existence of the minimax node system $\yy$ is also based on the properness property. In our case Theorem \ref{th:minimax} guarantees the existence of a minimax node, and we also know that it belongs to $X$. Consequently, we don't need the properness property of $\Gamma$ on $X$ -- very fortunately, because it is even false. (Nevertheless, $\Gamma$ \emph{is proper} on the \emph{whole set} $S$.) However, we need the following extension of the found sign rules from the minimax point $\yy$ to all $\xx \in X$.

\begin{lemma}\label{l:JiJkonX} If $\xx\in X$ and $1 \le i,k \le n+1$ then we have
$ (-1)^{i-k} D_i(\xx)D_k(\xx)>0$.
\end{lemma}
\begin{proof} First, if $i=k$, then the statement follows from the fact that $D_i(\xx)\ne 0$ at $\xx \in X$, which was given in Corollary \ref{c:nonsingularity}. For the following let us fix a pair $i,k$ with $1\le i\ne k \le n+1$. Again from the nonvanishing of these Jacobians, it follows that $(-1)^{i-k} D_i(\xx)D_k(\xx)\ne 0$ on $X$.

We know from Theorem \ref{th:minimax} the existence of at least one minimax node system $\yy \in X$. Lemma \ref{Shi_lemma2} gave that for any minimax node system $\yy \in X$ $(-1)^{k-i}D_i(\yy)D_k(\yy)>0$ holds true. Moreover, Lemma \ref{l:Xconnected} gave that $X\subset S$ is a connected open domain. Since the continuous image of a connected set is connected, $(-1)^{k-i}D_iD_k(X)$ is a connected subset of $\RR$, not containing 0, but having a positive element (namely, $(-1)^{k-i}D_i(\yy)D_k(\yy)$). So we must have $(-1)^{k-i}D_iD_k(X)\subset (0,\infty)$, that is, $(-1)^{k-i}D_i(\xx)D_k(\xx)>0$ for every $\xx\in X$.
\end{proof}

\emph{Continuation of the proof of Theorem \ref{t:homeom}.} Let $\xx\in X$ be arbitrary and let us consider the following homogeneous system of $n$ linear equations of $n+1$ unknown parameters $\lambda_i(\xx)$:
\begin{equation}\label{ineq_shi}
\sum\limits_{i=1}^{n+1} \lambda_i(\xx)\frac{\partial m_i}{\partial x_j}(\xx)=0, \quad j=1,\dots,n.
\end{equation}
Then, there exists a not identically zero solution vector $(\lambda_1(\xx),\dots,\lambda_{n+1}(\xx))$. So, e.g., let $\lambda_k(\xx)\ne 0$. Moreover, let us assume for normalization that $\lambda_k(\xx)=1$.

Rearranging the system \eqref{ineq_shi}, we obtain the non-homogeneous system of linear equations
\begin{equation}\label{ineq_shi_2_rearr}
\sum\limits_{i=1, ~i\ne k}^{n+1} \lambda_i(\xx)\frac{\partial m_i}{\partial x_j}(\xx)=-\frac{\partial m_k}{\partial x_j}(\xx), \quad j=1,\dots,n.
\end{equation}

Here we can repeat for a general $\xx \in X$ the argument in \eqref{lineq} -- \eqref{lambda_i JiJk}, given there only for a minimax point $\yy$. The coefficient matrix of \eqref{ineq_shi_2_rearr} is $A_k(\xx)$, hence its determinant is $D_k(\xx)\ne 0$. So Cramer's rule furnishes
\[
\lambda_i(\xx)=\frac{\det B_i^{(k)}(\xx)}{D_k(\xx)},
\]
where $B_i^{(k)}(\xx)$ can be obtained from $A_i(\xx)$ by replacing $-(\partial m_k/\partial x_1, \dots, \partial m_k/\partial x_n)^t=-\ba_k(\xx)$ for its $i$th column $\ba_i(\xx)$. Moreover, $A_i(\xx)$ can be obtained from $ B_i^{(k)}(\xx)$, by changing the column $-\ba_k(\xx)$ with its neighbor $|i-k|-1$ times to reach the $k$th place and then multiplying it by $(-1)$. It shows that $\det B_i^{(k)}(\xx)=(-1)^{i-k}\det A_i(\xx)=(-1)^{i-k} D_i(\xx)$.
Consequently, we get that $\lambda_i(\xx)=(-1)^{i-k}D_i(\xx)/D_k(\xx)$ for every $i\in\{1,\dots,n+1\}\setminus\{k\}$, which guarantees that $\lambda_i(\xx) \ne 0$, for every $i=1,\dots,{n+1}$.

Freeing ourselves from the normalization $\lambda_k(\xx)=1$, we obtain from Lemma  \ref{th:LP_Avriel} that any nontrivial solution $(\lambda_1,\ldots,\lambda_{n+1})$ of \eqref{ineq_shi} must have strictly nonzero coordinates with all of the same sign.

Next, let us consider the system
%\begin{equation}\label{lin_eq2}
%\begin{aligned}
%\sum\limits_{i=1}^{n+1} \lambda_i(\xx)&=0,
%\\
%\sum\limits_{i=1}^{n+1} \lambda_i(\xx)\frac{\partial m_i}{\partial x_j}(\xx)&=0, \quad j=1,\dots,n.
%\end{aligned}
%\end{equation}
\begin{equation}\label{lin_eq2}
\sum\limits_{i=1}^{n+1} \lambda_i(\xx)=0,
\qquad \text{and} \qquad
\sum\limits_{i=1}^{n+1} \lambda_i(\xx)\frac{\partial m_i}{\partial x_j}(\xx)=0, \quad j=1,\dots,n.
\end{equation}
For a fixed $\xx\in X$ this is a homogeneous system of $n+1$ linear equations of the $n+1$ unknowns $\lambda_i(\xx)$. Moreover, this system has only the trivial solution, as otherwise the first equation implies that there are both strictly negative and strictly positive coordinates of $(\lambda_1(\xx),\dots,\lambda_{n+1}(\xx))$, which contradicts to what has been established for solutions of the subsystem \eqref{ineq_shi}. So, this is a nonsingular system of linear equations.

Therefore, the determinant of the coefficient matrix of \eqref{lin_eq2} is nonvanishing, i.e.
\begin{equation}
\begin{aligned}
\left|\begin{array}{cccc}
1& 1& \ldots & 1\\
\frac{\partial m_1}{\partial x_1}(\xx) & \frac{\partial m_2}{\partial x_1}(\xx) & \ldots & \frac{\partial m_{n+1}}{\partial x_1}(\xx)\\
\vdots & \vdots &  \ddots & \vdots\\
\frac{\partial m_1}{\partial x_n}(\xx) & \frac{\partial m_2}{\partial x_n}(\xx) & \ldots & \frac{\partial m_{n+1}}{\partial x_n}(\xx)
\end{array}
\right|\not=0 .
\end{aligned}
\end{equation}
Subtracting from each column the preceding one, starting form the right and continuing until the second, we get the equivalent determinant
\begin{equation}
\begin{aligned}
\left|\begin{array}{cccc}
1 &  0 & \ldots &  0
\\
\frac{\partial m_1}{\partial x_1}(\xx)   & \frac{\partial m_2}{\partial x_1}(\xx)-\frac{\partial m_1}{\partial x_1}(\xx)   & \ldots & \frac{\partial m_{n+1}}{\partial x_1}(\xx)-\frac{\partial m_{n}}{\partial x_1}(\xx)
\\
\vdots & \vdots & \ddots & \vdots
\\
\frac{\partial m_1}{\partial x_n}(\xx)   & \frac{\partial m_2}{\partial x_n}(\xx)-\frac{\partial m_1}{\partial x_n}(\xx) & \ldots &\frac{\partial m_{n+1}}{\partial x_n}(\xx)-\frac{\partial m_{n}}{\partial x_n}(\xx)
\end{array}
\right| \ne 0.
\end{aligned}
\end{equation}
Finally, expanding the determinant with respect to its first row, we get that $\det(D\Gamma)\not=0$ on $X$, too.

Altogether it means, that $\det(D\Gamma)\not=0$ both on $X$ and on $X^c$. Therefore, $\Gamma$ is a local, and hence in view of properness, also a global homeomorphism.
\end{proof}

\begin{corollary}
The equioscillating node system, and therefore the extremal node system is necessarily unique.
\end{corollary}

\begin{proof}
For every equioscillating node system, its image under $\Gamma$ is the zero vector $\bf{0}$. However, $\Gamma$ is a homeomorphism, hence there exists a unique $\yy\in S$ such that $\Gamma(\yy)=\bf{0}$. Moreover Theorem \ref{th:minimax} has already given that a minimax node system must belong to $X$ and have the equioscillation property. Therefore, the unique equioscillating system is the only optimal node system.
\end{proof}

\section{The Erd\H os Conjecture and issues of majorization}\label{sec:majorization}

In this section we will prove that the Erd\H os Conjecture also holds for exponentially weighted polynomial interpolation on the halfline. However, somewhat surprisingly, we will also find that the strong intertwining of maxima vectors $\mm(\xx), \mm(\yy)$, (known to hold in the classical case, see Theorem 3 of de Boor and Pinkus \cite{CBoorPinkus}) fails to hold in this setting. We will give a precise description in which extent the  property generalizes to exponentially weighted interpolation. These findings underline the relevance of singularity issues -- once there are singular derivative matrices, we necessarily have a more complicated picture than we had classically.

\begin{lemma}\label{l:Cramer} Let $1\le k \le n+1$ and let $\Phi(\xx):=\Phi_k(\xx):=(m_i(\xx))_{i=1,i\ne k}^{n+1}$ ($\xx\in S$). Assume that $D_k(\uu)\ne 0$ for a node system $\uu \in S$. Then in a sufficiently small neighborhood $U \subset S$ of $\uu$, $\Phi:U\leftrightarrow W$ is a homeomorphism between $U$ and $W:=\Phi(U)$, and on $W$ $m_k$ can be expressed as a continuously differentiable function $G=G_k$ of the interval maxima $\ww:=(m_1,\ldots,m_{k-1},m_{k+1},\ldots, m_{n+1})\in W$. Moreover, for $\ww=\Phi(\vv)$ with $\vv \in U$ the gradient of $G$ is
\begin{equation}\label{Mnderivative}
\nabla_{\ww} G (\Phi(\vv)) =
\begin{bmatrix}
. \\ . \\ . \\
\frac{\partial m_k}{\partial m_i}(\ww)
\\ . \\ . \\ .
\end{bmatrix}_{i=1,i\ne k}^{n+1}
= \begin{bmatrix}
. \\ . \\ .\\
\frac{(-1)^{k-i-1}D_i(\vv)}{D_k(\vv)}
\\ . \\ . \\ .
\end{bmatrix}_{i=1, i\ne k}^{n+1}
(\vv \in U, \ww=\Phi(\vv)\in W).
\end{equation}
\end{lemma}
\begin{proof} We already know that all the $m_i$ are continuously differentiable functions of the variables $\xx \in S$, hence the function $$\Phi(\xx):= (m_1(\xx),\ldots,m_{k-1}(\xx),m_{k+1}(\xx),\ldots,m_{n+1}(\xx))$$ is continuously differentiable, too. Moreover, wherever $D_k(\xx)\ne 0$, it is a local homeomorphism with derivative matrix $A_k^T(\xx)$ at $\xx$. Therefore, it has a continuously differentiable inverse function $\Phi^{-1}$, too. So, we can write $m_k(\xx)=m_k(\Phi^{-1}(m_1(\xx),\ldots,m_{k-1}(\xx),m_{k+1}(\xx),\ldots,m_{n+1}(\xx)))$, and $G=m_k \circ \Phi^{-1}$ or $G \circ \Phi= m_k$. From this it is clear that $G$ is a continuously differentiable function of the values $(m_1,\ldots,m_{k-1},m_{k+1},\ldots,m_{n+1})$ in a small neighborhood of any given point $\Phi(\uu)=(m_i(\uu))_{i=1, i\ne k}^{n+1}$ with $D_k(\uu)\ne 0$. Equivalently, for a small neighborhood $U$ of $\uu \in S$, $G$ is a continuously differentiable mapping in $W=\Phi(U)$.

So let us differentiate at some $\vv \in U$ the composite function $m_k=G\circ \Phi$. By an application of the chain rule we get for each $j=1,\ldots,n$ that\footnote{To be fully precise, instead of the usual formalism for the chain rule, used here, as well as in the very statement of the Lemma, we should write $G$ in place of $m_k$, as formally the domain of $m_k$ is $S$, as it is the composite function $G\circ \Phi$, whereas here only $G$ is differentiated (with respect to the other $m_i$ in $W$). So, here $\sum_{i=1, i\ne k}^{n+1} \frac{\partial G}{\partial m_i}(\Phi(\vv)) \frac{\partial m_i}{\partial x_j}(\vv)$ should stand, but with a slight abuse of notation we write $m_k$ in place of $G=G_k$ here and everywhere.}
$$
\frac{\partial m_k}{\partial x_j}(\vv) = \sum_{i=1, i\ne k}^{n+1} \frac{\partial m_k}{\partial m_i}(\Phi(\vv)) \frac{\partial m_i}{\partial x_j}(\vv),
$$
or, in matrix form
\begin{equation}\label{Mnderivative_2}
\begin{bmatrix}
. \\ . \\ . \\
\frac{\partial m_k}{\partial x_j}(\vv)
\\ . \\ . \\ .
\end{bmatrix}_{j=1}^n
= \begin{bmatrix}
. & .& . \\ . & .& .  \\ . & .& . \\
. &  \frac{\partial m_i}{\partial x_j}(\vv) & . \\
. & .& . \\ . & .& .  \\ . & .& .
\end{bmatrix}_{j=1,i=1, i\ne k}^{n,n+1}
\cdot \quad
\begin{bmatrix}
. \\ . \\ . \\
\frac{\partial m_k}{\partial m_i}(\Phi(\vv))
\\ . \\ . \\ .
\end{bmatrix}_{i=1, i\ne k}^{n+1}
.
\end{equation}

The square matrix on the right hand side is $A_k(\vv)$, with determinant $D_k(\vv)\ne 0$. So, this system of equations is a linear equation in $\RR^n$ with nonsingular coefficient matrix $A_k$ and constants $\beta_{j}:=\frac{\partial m_k}{\partial x_j}(\vv)$, considering the partial derivatives $\xi_i:=\frac{\partial m_k}{\partial m_i}(\ww)=\frac{\partial m_k}{\partial m_i}(\Phi(\vv))$ as the $n$ unknowns.
By Cramer's rule, the system can be solved and the unknowns can be represented as
$$
\xi_i=\frac{\det B_i}{D_k},
$$
where $B_i$ is the matrix\footnote{Note that $B_i$ is almost the same as $B_i^{(k)}$ in \eqref{lambda_i}, but for the negative sign of $\ba_k$ in the latter, causing the opposite sign of $\xi_i$ here.} obtained from $A_k$ by replacing its $i$th column $\ba_i$ by the vector of constants $\bb:=(\beta_1,\ldots,\beta_n)^t$. However, this new $i$th column is just the $k$th column $\ba_k$ of $A$, and exchanging it one by one with the $i+1$st, $i+2$nd etc. $k-1$th columns if $k>i$ and $i-1$st, $i-2$nd etc. $k+1$th if $k<i$, we are led to $A_i$, the matrix obtained from $A$ by deleting its $i$th column. So, $\det B_i=(-1)^{k-i-1} \det A_i=(-1)^{k-i-1} D_i$ and $\xi_i=(-1)^{k-i-1}D_i/D_k$, as desired.
\end{proof}

\begin{corollary}\label{l:pert_mkmi}
Let $\xx\in X$. Then, for any fixed $k$ with $1\le k \le n+1$, in a small neighborhood $W$ of $(m_1(\xx),\ldots,m_{k-1}(\xx),m_{k+1}(\xx),\ldots,m_{n+1}(\xx))$ the interval maxima $m_k$ can be expressed as a differentiable function of $\ww=(m_1,\ldots,m_{k-1},m_{k+1},\ldots,m_{n+1})$, and $\partial m_k/ \partial m_i<0$ for every $i=1,\dots,n+1$, $i\ne k$.
\end{corollary}

\begin{proof}
We have just seen that $\frac{\partial m_k}{\partial m_i}(\ww) = (-1)^{k-i-1} D_i(\vv)/D_k(\vv)$, if $\vv$ belongs to a small neighborhood of $\xx$ and $\ww=\Phi(\vv)$. Comparing to Lemma \ref{l:JiJkonX} immediately furnishes the negativity of this expression.
\end{proof}

\begin{theorem}[Limitations to majorization]\label{th:nonmajorization}
If for any pair $\xx, \yy \in S$ we have that $m_i(\xx) \le m_i(\yy)$, then either $\xx=\yy$, or both $\xx, \yy \in X^c$.

Such majorizing pairs in $X^c$ do occur, however, even in the following strict sense: it may well happen that $m_i(\xx)<m_i(\yy)$ for $i=1,\ldots,n$, whereas $m_{n+1}(\xx)=m_{n+1}(\yy)=1$.
\end{theorem}

Before proving the theorem itself, let us present here that the first part of the theorem entails in particular the positive answer to the Erd\H os Conjecture in our setting.

\begin{corollary}[Sandwich Property]\label{cor:sandwich}
If $\yy$ is the unique extremal node system, then for any other node system $\xx \ne \yy$ we necessarily have
$$
\mul(\xx):=\min_{1\le i\le n+1} m_i(\xx)< m(\yy) <\mol(\xx):=\max_{1\le i\le n+1} m_i(\xx).
$$
\end{corollary}
\begin{proof} By Theorem \ref{th:minimax}, the unique extremal node system $\yy$ is situated in $X$, and equioscillates. Therefore $\overline{m}(\yy)=m_{n+1}(\yy)>1$. Now let $\xx \in X$ be a node system with $\mm(\xx)$ majorizing $\mm(\yy)$, or conversely, with $\mm(\yy)$ majorizing $\mm(\xx)$. Then according to Theorem \ref{th:nonmajorization} either $\xx=\yy$, which was excluded in the formulation of our Corollary, or we would have both $\xx, \yy \in X^c$, which contradicts to $\yy \in X$. Therefore, majorizing in any direction is impossible. The corollary is proved.
\end{proof}

\begin{proof}[Proof of Theorem \ref{th:nonmajorization}]
First, let $\xx=(x_1,\dots,x_n)$ belong to $X^c$, i.e., $z_{n+1}(\xx)=x_n$. Then with any $x_n'>x_n$ the node system  $\xx'=(x_1,\dots,x_{n-1},x_n')$ belongs to $ X^c$, too, in view of Lemma \ref{l:xstar}.

Moreover, Lemma \ref{l:partialmixn} (i.e., $\partial m_i/\partial x_n>0$ in $X^c$) entails that $m_i(\xx') > m_i(\xx)$ for $i=1,\dots,n$, whereas $m_{n+1}(\xx)=m_{n+1}(\xx')=1$. It follows that $\xx'$ majorises $\xx$, and the majorization is strict save the last coordinate $m_{n+1}$. Note that we found a majorant for every $\xx \in X^c$ with the same first coordinates $x_1,\ldots,x_{n-1}$. In particular, the second part of the assertion follows.

Now, assume that $m_i(\xx) \le m_i(\yy)$ such that $\xx$ and $\yy$ do not belong to $X^c$, simultaneously. It implies that $\yy\in X$, (otherwise $1=m_{n+1}(\yy)\ge m_{n+1}(\xx)\ge 1$ shows that $\xx\in X^c$, too).

First we prove that then either $\xx=\yy$, and we have equalities for every $i$, or there exists some strictly majorizing $\yy'\in X$ of $\xx$. Indeed, if $m_i(\yy)=m_i(\xx)$ for all indices $i$, then also $\Gamma(\xx)=\Gamma(\yy)$, hence by Theorem \ref{t:homeom} we must have $\xx=\yy$, too. Furthermore, if there exists an index $k$ with $m_k(\yy)>m_k(\xx)$, then referring to $D_k(\yy)\ne 0$ (for $\yy \in X$), we can slightly perturb the node system $\yy$ to some $\yy'$ with all values $m_i(\yy')$ with $i\ne k$ prescribed to strictly exceed $m_i(\xx)$, whereas  $m_k(\yy')$ can still be kept larger than $m_k(\xx)$.

So we can as well assume that $\yy$ strictly majorises $\xx$: $m_i(\yy)>m_i(\xx)$, $i=1,\ldots,n+1$. Denote $\mu:=\mu(\yy):=\min\limits_{1\le i\le n+1}(m_i(\yy)-m_i(\xx))$ and $M:=M(\yy):=\max\limits_{1\le i\le n+1}(m_i(\yy)-m_i(\xx))$. Strict majorization means $\mu>0$. Also, in general for arbitrary $\ww \in S$ denote
$$
\mu(\ww):=\min\limits_{1\le i\le n+1}(m_i(\ww)-m_i(\xx)), \qquad  M(\ww):=\max\limits_{1\le i\le n+1}(m_i(\ww)-m_i(\xx)).
$$
Let us consider the set of points $W:=\{\ww \in S~:~ \mu(\ww)\ge \mu, M(\ww)\le M\}$. As $\yy \in W$, this set is nonempty; by continuity of the functions $m_i$, the intersection of these closed level sets is also closed. Further, if $\ww \in W$, then $\mol(\ww)\le \mol(\xx)+M$, hence $W \subset \mm^{-1}([1,\mol(\xx)+M]^{n+1})$. The latter set is a compact subset of $S$, for $\mm$ is a proper mapping according to Corollary \ref{c:proper_m}. Therefore, its closed subset $W$ is also a compact subset of $S$. Also, $W \subset X$ because $\mu(\ww)\ge \mu>0$ entails $m_{n+1}(\ww) >1$. By compactness, the continuous function $\mu(\ww)$ has a finite maximum value $\mu^*$ on $W$, and $\mu$ must attain this maximum on $W$. So let us restrict attention to the maximal subset $Z\subset W$, where $\mu(\ww)=\mu^*$. As another closed level set, it is a closed, hence compact subset of $W$, which is not empty (in view of the existence of a $\mu$-maximal point in $W$). Let us now \emph{minimize} the value of $M(\ww)$ on $Z$. As before, on the compact set $Z$ there must exist some point(s) $\zz \in Z \subset W$ with $M(\zz)=M^*:=\min\{ M(\ww)~:~ \ww \in Z\}$.

Now we claim that $M(\zz)=\mu(\zz)$ for this node system. Before proving, let us see that it concludes the proof of Theorem \ref{th:nonmajorization}. Indeed, then we would arrive at a node system $\zz$, such that $m_i(\zz)=m_i(\xx)+\mu^*$ for all $i=1,\ldots,n$, hence $\Gamma(\zz)=\Gamma(\xx)$, so that taking into account that $\Gamma$ is a homeomorphism, also $\zz=\xx$, which is not possible as we had $\mu^* \ge \mu(\yy)>0$.

So it remains to prove $M^*=\mu^*$, i.e., $M(\zz)=\mu(\zz)$. Let us assume for a contradiction that $M(\zz)>\mu(\zz)$. Then write for the index sets of maximal and minimal differences $\I$ and $\JJ$:
\begin{equation}
\begin{aligned}
\I:=\{K~:~m_{K}(\zz)-m_{K}(\xx)=M^*\}; \qquad
%% \\[2mm]
\JJ:=\{k~:~ m_{k}(\zz)-m_{k}(\xx)=\mu^*\}.
\end{aligned}
\end{equation}
Then $\mu^*<M^*$ and $\I\cap \JJ=\emptyset$, but neither $\I$ nor $\JJ$  is empty, so that each set can have at most $n$ indices in them. Let us pick one $K\in \I$ and one $k\in \JJ$. Invoking Corollary \ref{l:pert_mkmi} we get  $\frac{\partial m_K}{\partial m_k}<0$, hence we can obtain a new node system $\zz'$ in $X$, arbitrarily close to $\zz$, such that $m_i(\zz)=m_i(\zz')$ for all $i\ne k, K$, while the corresponding $m_K(\zz')$ is a little smaller, than $m_K(\zz)$ and simultaneously $m_k(\zz')$ is a little greater than $m_k(\zz)$. Note that for the indices $i\ne k,K$, the difference $m_i(\zz')-m_i(\xx)$ remained equal to $m_i(\zz)-m_i(\xx)$, as $m_i$ was not changed.

Therefore, for the new node system $\zz'$ the respective values $\mu':=\mu(\zz') \ge \mu(\zz)$, and $M':=M(\zz') \le M(\zz)$, so that $\zz'\in W$, too, moreover, either $M'< M^*$, or at least $K$ is dropped from $\I$, and either $\mu'>\mu^*$, or at least $k$ is dropped from $\JJ$. Therefore, in at most $n$ steps we arrive at another node system $\zz'' \in X$, arbitrarily close to $\zz$, such that either $\mu(\zz'')>\mu(\zz)$ or $M(\zz'')<M(\zz)$. As the minimum is not decreased, and the maximum is not increased, the perturbed node system remains in the range where $m_i(\zz'')\in [m_i(\xx)+\mu^*,m_i(\xx)+M^*]$. That is, $\zz, \zz', \zz'' \in Z$. However, $\mu^*$ was maximal and $M^*$ was minimal possible, so the very existence of $\zz''\in Z$ yields a contradiction. Consequently, the equality $\mu(\zz)=M(\zz)$ must hold, and the theorem is proved.
\end{proof}

\end{document}